\documentclass[article]{IEEEtran}
\usepackage{cite}
\usepackage{color}
\usepackage[pdftex]{graphicx}
\graphicspath{{fig/}{jpeg/}}
\usepackage[cmex10]{amsmath}
\usepackage{amssymb}
\usepackage{algorithm}
\usepackage{algpseudocode}
\usepackage{amsmath,amssymb,lipsum}
\usepackage{hyperref}
\usepackage{comment}
\usepackage{graphicx}
\usepackage[font=small]{caption}
\usepackage{subcaption}
\input{mysymbol.sty}
\usepackage{needspace}

\usepackage{amsthm}
\usepackage{bbm}
\usepackage{tikz}
\usetikzlibrary{shapes,arrows}

\newtheorem{theorem}{Theorem}
\newtheorem{definition}{Definition}
\newtheorem{lemma}{Lemma}
\newtheorem{corollary}{Corollary}

\theoremstyle{definition}
\newtheorem{assumption}{Assumption}
\newtheorem{remark}{Remark}

\title{Online Learning of Feasible Strategies in\\ Unknown Environments}
\author{Santiago Paternain and Alejandro Ribeiro
\thanks{Work in this paper is supported by NSF CNS-1302222 and ONR N00014-12-1-0997. The authors are with the Department of Electrical and Systems Engineering, University of Pennsylvania, 200 South 33rd Street, Philadelphia, PA 19104. Email: \{spater, aribeiro\}@seas.upenn.edu.}}

\begin{document}

\maketitle
%\thispagestyle{empty}
%\pagestyle{empty}

%%%%%%%%%%%%%%%%%%%%%%%%%%%%%%%%%%%%%%%%%%%%%%%%%%%%%%%%%%%%%%%%%%%%%%%%%%%
%%%   A   B   S   T   R   A   C   T   %%%%%%%%%%%%%%%%%%%%%%%%%%%%%%%%%%%%%
%%%%%%%%%%%%%%%%%%%%%%%%%%%%%%%%%%%%%%%%%%%%%%%%%%%%%%%%%%%%%%%%%%%%%%%%%%%
%
\begin{abstract}
Define an environment as a set of convex constraint functions that vary arbitrarily over time and consider a cost function that is also convex and arbitrarily varying. Agents that operate in this environment intend to select actions that are feasible for all times while minimizing the cost's time average. Such action is said optimal and can be computed \textit{offline} if the cost and the environment are known a priori. An \textit{online} policy is one that depends causally on the cost and the environment. To compare online policies to the optimal offline action define the fit of a trajectory as a vector that integrates the constraint violations over time and its regret as the cost difference with the optimal action accumulated over time. Fit measures the extent to which an online policy succeeds in learning feasible actions while regret measures its success in learning optimal actions. This paper proposes the use of online policies computed from a saddle point controller which are shown to have fit and regret that are either bounded or grow at a sublinear rate. These properties provide an indication that the controller finds trajectories that are feasible and optimal in a relaxed sense. Concepts are illustrated throughout with the problem of a shepherd that wants to stay close to all sheep in a herd. Numerical experiments show that the saddle point controller allows the shepherd to do so.\end{abstract}

%%%%%%%%%%%%%%%%%%%%%%%%%%%%%%%%%%%%%%%%%%%%%%%%%%%%%%%%%%%%%%%%%%%%%%%%%%%
%%%   S   E   C   T   I   O   N   %%%%%%%%%%%%%%%%%%%%%%%%%%%%%%%%%%%%%%%%%
%%%%%%%%%%%%%%%%%%%%%%%%%%%%%%%%%%%%%%%%%%%%%%%%%%%%%%%%%%%%%%%%%%%%%%%%%%%
%
%!TEX root = feasibility_continous_time.tex

%%%%%%%%%%%%%%%%%%%%%%%%%%%%%%%%%%%%%%%%%%%%%%%%%%%%%%%%%%%%%%%%%%%%%%%%%%%%%% OLD %%%%%%%%%%%%%%%%%%%%%%%%%%%%%%%%%%%%%%%%%%%%%%%%%%%%%%%%%%%%%%%%%%%%%%%%%%%%%%%%%%%%%%%%%%%%%%%%%%%%%%%%%%%%
\section{Introduction}

The motivation for this paper is the navigation of a time varying convex environment defined as a set of convex constraints that an agent must satisfy at all times. The constraints are unknown a priori, vary arbitrarily in time in a possibly discontinuous manner, and are observed locally in space and causally in time. The goal of the agent is to find a feasible strategy that satisfies all of these constraints. This paper shows that an online version of the saddle point algorithm of Arrow and Hurwicz \cite{arrow_hurwicz} executed by the agent succeeds in finding such strategy. If the agent wants to further minimize a convex cost, we show that the same algorithm succeeds in finding an strategy that is feasible at all times and optimal on average.

To understand the contribution of this paper it is important to observe that the navigation problem outlined above can be mathematically formulated as the solution of a convex program %\cite{rimon1992exact,warren1989global,Khatib:1986:ROA:6806.6812,masoud2000constrained,ge2000new,vadakkepat2000evolutionary,moase2009newton} 
\cite{rimon1992exact,warren1989global,Khatib:1986:ROA:6806.6812,ge2000new,vadakkepat2000evolutionary}
whose solution is progressively more challenging when we progress from deterministic settings to stochastic and online settings. Indeed, in a determinist setting the cost and constraints are fixed. This yields a canonical convex optimization problem that can be solved with extremum seeking controllers based on gradient descent \cite{hirsch2004differential, krstic2000stability, ariyur2003real, tan2006non}, primal-dual methods \cite{arrow_hurwicz,nedic2009subgradient,uzawa1958iterative,maistroskii1977gradient,feijer2010stability}, or interior point methods \cite[Chapter 11]{boyd2004convex}. In a stochastic setting cost and constraints are not constant but vary randomly according to a stationary distribution. The agent's goal is then expressed as the selection of an action that minimizes the expected value of the objective function while satisfying constraints in an average sense \cite{ atanasov2012stochastic, azuma2012stochastic, Liu20101443}
%\cite{azuma2010nonholonomic, atanasov2012stochastic, azuma2012stochastic, Liu20101443}. 
This problem is more complicated than its deterministic counterpart but it can be solved using, e.g., stochastic gradient descent \cite{robbins1951stochastic, schmidt2013minimizing, konevcny2013semi} or stochastic quasi-Newton's methods \cite{mokhtari2014res}.

In this paper we consider online formulations in which cost and constraints can vary arbitrarily, perhaps strategically, and where the goal is to find an action that is good on average and that satisfies the constraints at all times -- assuming such an action exists, which, when functions change strategically, restricts adversarial actions. In this case, {\it unconstrained} cost minimization can be formulated in the language of regret \cite{blackwell1956analog, vapnik2000nature, shalev2011online} whereby agents operate online by selecting plays that incur a cost selected by nature. The cost functions are revealed to the agent ex post and used to adapt subsequent plays. The goodness of these {\it online} policies are determined by comparing to the optimal action chosen \textit{offline} by a clairvoyant agent that has prescient access to the cost. Regret is defined as the difference of the accumulated cost attained online and the optimal offline cost. It is a remarkable fact that an online version of gradient descent is able to find plays whose regret grows at a sublinear rate when the cost is a convex function \cite{Zinkevich03, hazan2007logarithmic} -- therefore suggesting vanishing per-play penalties of online plays with respect to the clairvoyant play. 

The constrained optimization equivalent of gradient descent is the saddle point method applied to the determination of a saddle point of the Lagrangian function \cite{arrow_hurwicz}. This method interprets each constraint as a separate potential and descends on a linear combination of their gradients. The coefficients of this linear combination are multipliers that adapt dynamically so as to push the agent to the optimal solution in the feasible region. Saddle point algorithms and variations have been widely studied \cite{nedic2009subgradient,uzawa1958iterative,maistroskii1977gradient,feijer2010stability} and used in various domains such as decentralized control \cite{low1999optimization,chiang2007layering} and image processing, see e.g. \cite{chambolle2011first}. Our observation is that since an online version of gradient descent succeeds in achieving small regret, it is not unreasonable to expect an online saddle point method to succeed in finding feasible actions with small regret.

The main contribution of this paper is to prove that this expectation turns out to be true. We show that an online saddle point algorithm that observes costs and constraints ex post succeeds in finding policies that are feasible and have small regret. Central to this development is the definition of a viable environment as one in which there exist an action that satisfies the time varying constraints at all times and the introduction of the notion of fit (Section \ref{sec:survivability}). The latter is defined as a vector that contains the time integrals of the constraints evaluated across the trajectory and is the analogous of regret for the satisfaction of constraints. In the same way in which the accumulated payoff of the online trajectory is compared with the payoff of the offline trajectory, fit compares the accumulation of the constraints along the trajectory with the feasibility of an offline viable strategy. As such, a trajectory can achieve small fit by becoming feasible at all times or by alternating periods in which the constraints are violated with periods in which the constraints are satisfied with slack. This notion of fit is appropriate for constraints that have a cumulative nature. For cases where this is not appropriate we introduce the notion of saturated fit in which only violations of the constraint are accumulated. A trajectory with small saturated fit is one in which the constraints are violated by a significant amount only for a short period of time.

Technical developments begin with the derivation of a projected gradient controller to limit the growth of regret in an environment without constraints (Section \ref{sec:continuous_regret}). The purpose of this section is to introduce tools and to clarify connections with existing literature in discrete time \cite{Zinkevich03, hazan2007logarithmic} and continuous time regret \cite{viossat2013no, sorin2009exponential, kwon2014continuous}. An important conclusion here is that regret in continuous time can be bounded by a constant that is independent of the time horizon, as opposed to the sublinear growth that is observed in discrete time. 

We then move onto the main part of the paper in which we propose to control fit and regret growth with the use of an online saddle point controller that moves along a linear combination of the negative gradients of the instantaneous constraints and the objective function. The coefficients of this linear combination are adapted dynamically as per the instantaneous constraint functions (Section \ref{sec:main}). This online saddle point controller is a generalization of (offline) saddle point in the same sense that an online gradient controller generalizes (offline) gradient descent. We show that if there exists an action that satisfies the environmental constraints at all times, the online saddle point controller achieves bounded fit if optimality is not of interest (Theorem \ref{theo:not_opti}). When optimality is considered, the controller achieves bounded regret and a fit that grows sublinearly with the time horizon (Theorem \ref{theo:opti}). Analogous results are derived for saturated fit. I.e., it is bounded by a constant when optimality is not of interest and grows sublinearly otherwise (corollaries \ref{corollary_saturated_fit} and \ref{corollary_saturated_fit2}). Throughout the paper we illustrate concepts with the problem of a shepherd that has to stay close to his herd (Section \ref{sec_shepherd_problem}). A numerical analysis of this problem closes the paper (Section \ref{sec:examples}) except for concluding remarks (Section \ref{sec_conclusions}).

\medskip\noindent{\bf Notation.} A multivalued function $f:\reals^n\to\reals^m$ is defined by stacking component functions, i.e., $f:=[f_1,\ldots,f_m]^T$. The notation $\int f(x)dx:=[\int f_1(x)dx,\ldots,\int f_m(x)dx]^T$ represents a vector stacking individual integrals. An inequality $x\leq y$ between vectors $x,y\in\reals^n$ is interpreted componentwise. An inequality $x\leq c$ between a vector $x=[x_1,\ldots,x_n]^T\in\reals^n$ and a scalar $c\in\reals$ means that $x_i\leq c$ for all $i$.

%%%%%%%%%%%%%%%%%%%%%%%%%%%%%%%%%%%%%%%%%%%%%%%%%%%%%%%%%%%%%%%%%%%%%%%%%%%
%%%   S   E   C   T   I   O   N   %%%%%%%%%%%%%%%%%%%%%%%%%%%%%%%%%%%%%%%%%
%%%%%%%%%%%%%%%%%%%%%%%%%%%%%%%%%%%%%%%%%%%%%%%%%%%%%%%%%%%%%%%%%%%%%%%%%%%
%
\section{Viability, feasibility and optimality}\label{sec:survivability}
We consider a continuous time environment in which an agent selects actions that result in a time varying set of penalties. Use $t$ to denote time and let $X\subseteq \mathbb{R}^n$ be a closed convex set from which the agent selects action $x\in X $. The penalties incurred at time $t$ for selected action $x$ are given by the value $f(t,x)$ of the vector function $f:\reals\times\reals^n \to \reals^m$. We interpret the vector penalty function $f$ as a definition of the environment. Our interest is in situations where the agent is faced with an environment $f$ and must choose an action $x\in X$ -- or perhaps a trajectory $x(t)$ -- that guarantees nonpositive penalties $f(t,x(t))\leq 0$ for all times $t$ not exceeding a time horizon $T$. Since the existence of this trajectory depends on the specific environment we define a viable environment as one in which it is possible to select an action with nonpositive penalty for times $0\leq t\leq T$ as we formally specify next.

%%%%%%%%%%%%%%%%%%%%%%%%%%%%%%%%%%%%%%%%%%%%%%%%%%%%%%%%%%%%%%%%%%%%%%%%%%%
%%%   D   E   F   I   N   I   T   I   O   N   %%%%%%%%%%%%%%%%%%%%%%%%%%%%%
%%%%%%%%%%%%%%%%%%%%%%%%%%%%%%%%%%%%%%%%%%%%%%%%%%%%%%%%%%%%%%%%%%%%%%%%%%%
%
\begin{definition}[\textbf{Viable environment}]\label{def_viable_environment}
We say that an environment $f:\reals\times\reals^n \to \reals^m$ is viable over the time horizon $T$ for an agent that selects actions $x\in X$ if there exists a feasible action $x^\dagger\in X$ such that 
\begin{equation}\label{eqn_def_viable_environment}
   f(t,x^\dagger) \leq 0, \quad \text{for all\ }  t \in[0,T] .
\end{equation}
The set $X^\dagger:=\{x^\dagger \in X: f(t,x^\dagger) \leq 0,\ \text{for all\ } t \in[0,T]\}$ is termed the feasible set of actions.
\end{definition}

%%%%%%%%%%%%%%%%%%%%%%%%%%%%%%%%%%%%%%%%%%%%%%%%%%%%%%%%%%%%%%%%%%%%%%%%%%%
%%%   M   A   I   N       M   A   T   T   E   R   %%%%%%%%%%%%%%%%%%%%%%%%%
%%%%%%%%%%%%%%%%%%%%%%%%%%%%%%%%%%%%%%%%%%%%%%%%%%%%%%%%%%%%%%%%%%%%%%%%%%%
%
Since for a viable environment it is possible to have multiple feasible actions it is desirable to select one that is optimal with respect to some criterion of interest. Introduce then the objective function $f_0:\reals\times\reals^n \to \reals$, where for a given time $t \in [0,T]$ and action $x\in X$ the agent suffers a loss $f_0(t,x)$. The optimal action is defined as the one that minimizes the accumulated loss $\int_0^T f_0(t,x) \,dt$ among all viable actions, i.e.,
\begin{alignat}{2}\label{eqn_optimal_strategy}
   x^* :=\ &\argmin_{x\in X} \int_0^T f_0(t,x) \,dt \\ \nonumber
          &\st\ f(t,x) \leq 0,\ \text{for all\ } t \in [0,T] .
\end{alignat}
For the definition in \eqref{eqn_optimal_strategy} to be valid the function $f_0(t,x)$ has to be integrable with respect to $t$. In subsequent definitions and analyses we also require integrability of the environment $f$ as well as convexity with respect to $x$ as we formally state next.

%%%%%%%%%%%%%%%%%%%%%%%%%%%%%%%%%%%%%%%%%%%%%%%%%%%%%%%%%%%%%%%%%%%%%%%%%%%
%%%   A   S   S   U   M   P   T   I   O   N   %%%%%%%%%%%%%%%%%%%%%%%%%%%%%
%%%%%%%%%%%%%%%%%%%%%%%%%%%%%%%%%%%%%%%%%%%%%%%%%%%%%%%%%%%%%%%%%%%%%%%%%%%
%
\begin{assumption} \label{as:integrability}
The functions $f(t,x)$ and $f_0(t,x)$ are integrable with respect to $t$ in the interval $[0,T]$. 
\end{assumption}

\begin{assumption}\label{as:convexity}
The functions $f(t,x)$ and $f_0(t,x)$ are convex with respect to $x$ for all times $t\in[0,T]$. 
\end{assumption}

%%%%%%%%%%%%%%%%%%%%%%%%%%%%%%%%%%%%%%%%%%%%%%%%%%%%%%%%%%%%%%%%%%%%%%%%%%%
%%%   M   A   I   N       M   A   T   T   E   R   %%%%%%%%%%%%%%%%%%%%%%%%%
%%%%%%%%%%%%%%%%%%%%%%%%%%%%%%%%%%%%%%%%%%%%%%%%%%%%%%%%%%%%%%%%%%%%%%%%%%%
%
If the environment $f(t,x)$ and functions $f_0(t,x)$ are known beforehand, finding the action in a viable environment that minimizes the total aggregate cost is equivalent to solving the convex optimization problem in \eqref{eqn_optimal_strategy} for which a number of algorithms are known. Here, we consider the problem of adapting a strategy $x(t)$ when the functions $f(t,x)$ and $f_0(t,x)$ are {\it arbitrary} and {\it revealed causally.} I.e., we want to choose the action $x(t)$ using observations of viability $f(t,x)$ and cost $f_0(t,x)$ in the open interval $[0,t)$. This implies that $f(t,x(t))$ and $f_0(t,x(t))$ are not observed before choosing $x(t)$. The action $x(t)$ is chosen ex ante and the corresponding viability $f(t,x(t))$ and cost $f_0(t,x(t))$ are incurred ex post. Further observe that the constraints and objective functions may change abruptly if the number of discontinuities in these are finite for finite $T$. This makes the problem different from time varying optimization in which the goal is to track the optimal argument of $f_0(t,x)$ subject to the constraint $f(t,x)\leq 0$ under the assumption that these functions change continuously and at a sufficiently small rate %\cite{popkov2005gradient,fazlyab2015interior,zavala2010real,ling2013decentralized}. 
\cite{popkov2005gradient,fazlyab2015interior,zavala2010real}.

%%%%%%%%%%%%%%%%%%%%%%%%%%%%%%%%%%%%%%%%%%%%%%%%%%%%%%%%%%%%%%%%%%%%%%%%%%%
%%%   S   E   C   T   I   O   N   %%%%%%%%%%%%%%%%%%%%%%%%%%%%%%%%%%%%%%%%%
%%%%%%%%%%%%%%%%%%%%%%%%%%%%%%%%%%%%%%%%%%%%%%%%%%%%%%%%%%%%%%%%%%%%%%%%%%%
%
\subsection{Regret and fit}
We evaluate the performance of trajectories $x(t)$ through the concepts of regret and fit. To define regret we compare the accumulated cost $\int_0^T f_0(t,x(t)) \,dt$ incurred by $x(t)$ with the cost incurred by the optimal action $x^*$ defined in \eqref{eqn_optimal_strategy},
\begin{equation}\label{eqn_continuous_regret}
   \ccalR_T := \int_0^T f_0(t,x(t)) \,dt - \int_0^T f_0(t,x^*) \,dt .
\end{equation}
Analogously, we define the fit of the trajectory $x(t)$ as the accumulated penalties $f(t,x(t))$ incurred for times $t\in[0,T]$,
\begin{equation}\label{eqn_def_fit}
   \ccalF_{T} := \int_0^T f(t,x(t)) \,dt.
\end{equation}
The regret $\ccalR_T$ and fit $\ccalF_{T}$ can be interpreted as performance losses associated with online causal operation as opposed to offline clairvoyant operation. If $\ccalF_{T}$ is positive in a viable environment we are in a situation in which, had the environment be known a priori, we could have selected an action $x^\dagger$ with $f(t,x^\dagger) \leq 0$. The fit measures how far the trajectory $x(t)$ comes from achieving that goal. As in the case of the fit, if the regret $\ccalR_T$ is large we are in a situation in which prior knowledge of environment and cost would had resulted in the selection of the action $x^*$ -- and in that sense $\ccalR_T$ indicates how much we regret not having had that information available. 

Because of the cumulative nature of fit, it is possible to achieve small fit by alternating between actions for which the constraint functions take positive and negative values. This is valid when cumulative constraints are an appropriate model, which happens for quantities that can be stored or preserved in some sense -- such as energy budgets enforced through average power constraints. For situations where this is not appropriate, we define the saturated fit in which constraint slacks are saturated to a small constant $\delta$. Formally, let $\delta>0$ be a positive constant and define the function $\bar{f}_{\delta}(t,x)) = \max\left\{f(t,x),-\delta \right\}$. Then, the $\delta$-saturated fit is defined as
\begin{equation}\label{eqn_saturated_fit}
\bar\ccalF_{T} = \int_0^T \bar{f}_{\delta}(t,x(t)) \, dt.
\end{equation}
Since $\bar{f}_{\delta}(t,x)$ is the pointwise maximum of two convex functions with respect to the actions, it is a convex function itself and $\bar\ccalF_{T}$ is not different than the fit for the environment defined by $\bar{f}_{\delta}(t,x)$. By taking small values of $\delta$ we can reduce the negative portion of the fit to be as small as desired.  

A good learning strategy is one in which $x(t)$ approaches $x^*$. In that case, the regret and fit grow for small $T$ but eventually stabilize or, at worst, grow at a sublinear rate. Considering regret $\ccalR_T$ and fit $\ccalF_{T}$ separately, this observation motivates the definitions of feasible trajectories strongly feasible trajectories, and strong optimal trajectories that we formally state next.

%%%%%%%%%%%%%%%%%%%%%%%%%%%%%%%%%%%%%%%%%%%%%%%%%%%%%%%%%%%%%%%%%%%%%%%%%%%
%%%   D   E   F   I   N   I   T   I   O   N   %%%%%%%%%%%%%%%%%%%%%%%%%%%%%
%%%%%%%%%%%%%%%%%%%%%%%%%%%%%%%%%%%%%%%%%%%%%%%%%%%%%%%%%%%%%%%%%%%%%%%%%%%
%
\begin{definition}\label{def_0ptimality_and_viability}
Given an environment $f:\reals\times\reals^n \to \reals^m$, a cost $f_0:\reals\times\reals^n \to \reals$, and a trajectory $x(t)$ we say that:

\begin{mylist}
\item[\bf Feasibility.] The trajectory $x(t)$ is feasible in the environment if the fit $\ccalF_{T}$ grows sublinearly with $T$. I.e., if there exist a function $h(T)$ with $\limsup_{T\to\infty} h(T)/T = 0$ and a constant vector $C$ such that for all times $T$ it holds,
\begin{equation}\label{eqn_def_weak_survival}
   \ccalF_{T} := \int_0^T f(t,x(t)) \,dt \leq C h(T).
\end{equation} 

\item[\bf Strong Feasibility.] The trajectory $x(t)$ is strongly feasible in the environment if the fit $\ccalF_{T}$ is bounded for all $T$. I.e., if there exists a constant vector $C$ such that for all times $T$ it holds,
\begin{equation}\label{eqn_def_survival}
   \ccalF_{T} := \int_0^T f(t,x(t)) \,dt \leq C.
\end{equation} 

\item[\bf Strong optimality.] The trajectory $x(t)$ is strongly optimal in the environment if the regret $\ccalR_{T}$ is bounded for all $T$. I.e., if there exists a constant $C$ such that for all times $T$ it holds,
\begin{equation}\label{eqn_def_0ptimality}
   \ccalR_{T} := \int_0^T f_0(t,x(t)) \,dt - \int_0^T f_0(t,x^*) \,dt  \leq C.
\end{equation}

\end{mylist} \end{definition}

%%%%%%%%%%%%%%%%%%%%%%%%%%%%%%%%%%%%%%%%%%%%%%%%%%%%%%%%%%%%%%%%%%%%%%%%%%%
%%%   M   A   I   N       M   A   T   T   E   R   %%%%%%%%%%%%%%%%%%%%%%%%%
%%%%%%%%%%%%%%%%%%%%%%%%%%%%%%%%%%%%%%%%%%%%%%%%%%%%%%%%%%%%%%%%%%%%%%%%%%%
%
Having the regret satisfy $\ccalR_{T}\leq C$ irrespectively of $T$ is an indication that $f_0(t,x(t))$ is close to $f_0(t,x^*)$ so that the integral stops growing. This is not necessarily so because we can also achieve small regret by having $f_0(t,x(t))$ oscillate above and below $f_0(t,x^*)$ so that positive and negative values of $f_0(t,x(t)) - f_0(t,x^*)$ cancel out. In general, the possibility of having small regret by a trajectory that does not approach $x^*$ is a limitation of the concept of regret. Alternatively, we can think of the optimal offline policy $x^*$ as fixing a budget for cost accumulated across time. An optimal online policy meets that budget up to a constant $C$ -- perhaps by overspending at some times and underspending at some other times.

Likewise, when the fit satisfies $\ccalF_{T}\leq C$ irrespectively of $T$, it suggests that $x(t)$ approaches the feasible set. This need not be true as it is possible to achieve bounded fit by having $f(t,x(t))$ oscillate around $0$. Thus, as in the case of regret, we can interpret strongly feasible trajectories as meeting the {\it accumulated} budget $\int_0^T f(t,x(t))\,dt\leq 0$ up to a constant term $C$. This is in contrast with feasible actions $x^\dagger$ that meet the budget $f(t,x^\dagger)\leq0$ for all times. Feasible trajectories differ from strongly feasible trajectories in that the fit is allowed to grow at a sublinear rate. This means that feasible trajectories do not meet the {\it accumulated} budget within a constant $C$ but do meet the {\it time averaged} budget $(1/T)\int_0^T f(t,x(t))\,dt\leq0$ within that constant. The notion of optimality -- as opposed to strong optimality -- could have been defined as a case in which regret is bounded by a sublinear function of $T$. This is not necessary here because our results state strong optimality.

In this work we solve three different problems: (i) Finding strongly optimal trajectories in unconstrained environments. (ii) Finding strongly feasible trajectories. (iii) Finding feasible, strongly optimal trajectories. We develop these solutions in sections \ref{sec:continuous_regret}, \ref{subsec:non_opti}, and \ref{subsec:otpi}, respectively. Before that, we present two pertinent remarks and we clarify concepts with the introduction of an example.

%%%%%%%%%%%%%%%%%%%%%%%%%%%%%%%%%%%%%%%%%%%%%%%%%%%%%%%%%%%
%%%%%%%%%%%%%%%%%%%% R E M A R K %%%%%%%%%%%%%%%%%%%%%%%%%%%%%%%%
%%%%%%%%%%%%%%%%%%%%%%%%%%%%%%%%%%%%%%%%%%%%%%%%%%%%%%%%%%%
%
\begin{remark}[\bf Not every trajectory is strongly feasible]
In definition \eqref{eqn_def_survival} we consider the integral of a measurable function in a finite interval, hence it is always bounded by a constant. Yet if the latter depends on the time horizon $T$, the trajectory is not strongly feasible, because it is not uniformly bounded for all time horizons $T$. The same remark is valid for the definitions of strongly optimal and feasible. 
\end{remark}

\begin{remark}[\bf Connection with Stochastic Optimization]
One can think about the online learning framework as a generalization of the stochastic optimization setting (see e.g. \cite{robbins1951stochastic,borkar2008stochastic}). In the latter, the objective and constraint functions depend on a random vector $\theta \in \mathbb{R}^p$. Formally, the cost is a function $f_0: \mathbb{R}^n\times \mathbb{R}^p \to \mathbb{R}$ and the constraints are given by a multivalued function $f: \mathbb{R}^n\times \mathbb{R}^p \to \mathbb{R}^m$. The constrained stochastic optimization problem can be then formulated as 
\begin{equation}\label{eqn_stochastic_problem}
\begin{split}
x^*:=&\argmin \, \E{f_0(x,\theta)} \\
&\;\, \mbox{s.t.} \quad \quad \E{f(x,\theta)} \leq 0,
\end{split}
\end{equation}
where the above expectations are with respect to the random vector $\theta$. When the process that determines the temporal evolution of the random vector $\theta_t$ is stationary, the expectations can be replaced by time averages. In that sense  problem \eqref{eqn_stochastic_problem} is equivalent to the problem of generating trajectories that are feasible and optimal in the sense of Definition \ref{def_0ptimality_and_viability}.
\end{remark}

%%%%%%%%%%%%%%%%%%%%%%%%%%%%%%%%%%%%%%%%%%%%%%%%%%%%%%%%%%%%%%%%%%%%%%%%%%%
%%%   S   E   C   T   I   O   N   %%%%%%%%%%%%%%%%%%%%%%%%%%%%%%%%%%%%%%%%%
%%%%%%%%%%%%%%%%%%%%%%%%%%%%%%%%%%%%%%%%%%%%%%%%%%%%%%%%%%%%%%%%%%%%%%%%%%%
%
\subsection{The shepherd problem}\label{sec_shepherd_problem}

Consider a target tracking problem in which an agent -- the shepherd -- follows a group of $m$ targets -- the sheep. Specifically, let $z(t) = [z_1(t),z_2(t)]^T \in\reals^2$ denote the position of the shepherd at time $t$. To model smooth paths for the shepherd introduce a polynomial parameterization so that each of the position components $z_k(t)$ can be written as
\begin{equation}\label{eqn_shepherd_position}
   z_k(t) = \sum_{j=0}^{n-1} x_{kj} p_j(t),
\end{equation}
where $p_j(t)$ are polynomials that parameterize the space of possible trajectories. The action space of the shepherd is then given by the vector $x=[x_{10},\ldots,x_{1,n-1},x_{20},\ldots,x_{2,n-1}]^T\in\reals^{2n}$ that stacks the coefficients of the parameterization in \eqref{eqn_shepherd_position}. 

Further define $y_i(t)=[y_{i1}(t), y_{i2}(t)]^T$ as the position of the $i$th sheep at time $t$ for $i=1,\ldots,m$ and introduce a maximum allowable distance $r_i$ between the shepherd and each of the sheep . The goal of the shepherd is to find a path $z(t)$ that is within distance $r_i$ of sheep $i$ for all sheep. This can be captured by defining an $m$-dimensional environment $f$ with each component function $f_i$ defined as
\begin{equation}\label{eqn_sheep_environment}
   f_i(t,x) = \| z(t) - y_i(t)   \|^2 - r_i^2  \quad \mbox{for all} \quad i=1..m.
\end{equation}
That the environment defined by \eqref{eqn_sheep_environment} is viable means that it is possible to select a vector of coefficients $x$ so that the shepherd's trajectory given by \eqref{eqn_shepherd_position} stays close to all sheep for all times. To the extent that \eqref{eqn_shepherd_position} is a loose parameterization -- we can approximate arbitrary functions with sufficiently large index $n$, if the time horizon is fixed and not allowed to tend to infinity --, this simply means that the sheep are sufficiently close to each other at all times. E.g., if $r_i=r$ for all times, viability is equivalent to having a maximum separation between sheep smaller than $2r$. 

As an example of a problem with an optimality criterion say that the first target -- the black sheep -- is preferred in that the shepherd wants to stay as close as possible to it. We can accomplish that by introducing the objective function
\begin{equation}\label{eqn_black_sheep}
   f_0(t,x) = \| z(t) - y_1(t)   \|^2  .
\end{equation}
Alternatively, we can require the shepherd to minimize the work required to follow the sheep. This behavior can be induced by minimizing the integral of the acceleration which in turn can be accomplished by defining the optimality criterion [cf. \eqref{eqn_optimal_strategy}], 
\begin{equation}\label{eqn_minimum acceleration}
   f_0(t,x) = \big\|\ddot z(t)\big\| 
            = \Bigg\| \bigg[ \sum_{j=0}^{n-1} x_{1j} \ddot p_j(t),\
                         \sum_{j=0}^{n-1} x_{2j} \ddot p_j(t) \bigg] \Bigg\|.
\end{equation}
Trajectories $x(t)$ differ from actions in that they are allowed to change over time, i.e., the constant values $x_{kj}$ in \eqref{eqn_shepherd_position} are replaced by the time varying values $x_{kj}(t)$. A feasible or strongly feasible trajectory $x(t)$ means that the shepherd is repositioning to stay close to all sheep. An optimal trajectory with respect to \eqref{eqn_black_sheep} is one in which he does so while staying as close as possible to the black sheep. An optimal trajectory with respect to \eqref{eqn_minimum acceleration} is one in which the work required to follow the sheep is minimized. In all three cases we apply the usual caveat that small fit and regret may be achieved with stretches of underachievement following stretches of overachievement. 
%
%Notice that in general it is not possible to approximate the shepherd's position with a finite degree polynomial unless the time horizon is fixed and not allowed to tend to infinity, which is the case in the simulations presented in Section \ref{sec:examples}.
%
%%%%%%%%%%%%%%%%%%%%%%%%%%%%%%%%%%%%%%%%%%%%%%%%%%%%%%%%%%%%%%%%%%%%%%%%%%%
%%%   S   E   C   T   I   O   N   %%%%%%%%%%%%%%%%%%%%%%%%%%%%%%%%%%%%%%%%%
%%%%%%%%%%%%%%%%%%%%%%%%%%%%%%%%%%%%%%%%%%%%%%%%%%%%%%%%%%%%%%%%%%%%%%%%%%%
%
\section{Unconstrained regret in continuous time. }\label{sec:continuous_regret}
Before considering the feasibility problem we consider the following unconstrained minimization problem. Given an unconstrained environment ($f(t,x) \equiv 0$) our goal is to generate strong optimal trajectories $x(t)$ in the sense of Definition \ref{def_0ptimality_and_viability}, selecting actions from a closed convex set $X$, i.e., $x(t) \in X$ for all $t\in [0,T]$. Given the convexity of the objective function with respect to the action, as per Assumption \ref{as:convexity}, it is natural to consider a gradient descent controller. To avoid restricting attention to functions that are differentiable with respect to $x$, we  work with subgradients. For a convex function $g:X\to \mathbb{R}$ a subgradient $g_x$ satisfies the 
\begin{equation}\label{eqn_def_subgradient}
   g(y) \geq g(x) + g_x(x)^T(y-x) \quad \mbox{for all} \quad y\in X.
\end{equation}
In general, subgradients are defined at all points for all convex functions. At the points where the function $f$ is differentiable the subgradient and the gradient coincide. In the case of vector functions $f:\mathbb{R}^n \rightarrow \mathbb{R}^m$ we group the subgradients of each component into a matrix $f_x(x)\in\reals^{n\times m}$ defined as
\begin{equation}\label{eqn_subgradient_f}
   f_x(x) = \left[ \begin{array}{c c c c} f_{1,x}(x)  & f_{2,x}(x) & \cdot\cdot\cdot & f_{m,x}(x) \end{array}\right],
\end{equation}
where $f_{i,x}(x)$ is a subgradient of $f_i(x)$. 
In addition, since the action must always be selected from the set $X$ we define the controller in a way that the actions are the solution of a projected dynamical system over the set $X$. The solution has been studied in %\cite{cojocaru2004existence,Zhang95} 
\cite{Zhang95} and we define the notion as follow. 
%
%%%%%%%%%%%%%%%%%%%%%%%%%%%%%%%%%%%%%%%%%%%%%%%%%%%%%%%%%%%%%%%%%%%%%%%%%%%
%%%   D   E   F   I   N   I   T   I   O   N   %%%%%%%%%%%%%%%%%%%%%%%%%%%%%
%%%%%%%%%%%%%%%%%%%%%%%%%%%%%%%%%%%%%%%%%%%%%%%%%%%%%%%%%%%%%%%%%%%%%%%%%%%
%
\begin{definition}[Projected dynamical system]\label{def_projected_dynamical_system}
Let $X$  be a closed convex set.

\begin{mylist}
\item[\bf Projection of a point.] For any $z \in R^n$, there exits a unique element in $X$, denoted $P_X(z)$ such that
\begin{equation}\label{eqn_proj_over_X}
P_X(z) = \argmin_{y \in X} \|y-z \|.
\end{equation}

\item[\bf Projection of a vector at a point.]Let $x \in X$ and $v$ a vector, the projection of $v$ over the set $X$ at the point $x$ is
\begin{equation}
\Pi_X(x,v) = \lim_{\delta \to 0^+}\left(P_X(x+\delta v) -x\right) / \delta.
\end{equation}
%As it is demonstrated in Lemma \ref{lemma_tangent_cone}, the projection of a vector at a point over a set is equivalent to project the vector over the smallest cone containing the set $X$ with vertex at the point $x$. 

\item[\bf Projected dynamical system.]Given a closed convex set $X$ and a vector field $F(t,x)$ which takes elements from $\mathbb{R}\times X$ into $\reals^n$ the projected differential equation associated with $X$ and $F$ is defined to be
\begin{equation}
\dot{x}(t) = \Pi_X\left(x,F(t,x)\right).
\end{equation}
\end{mylist}
\end{definition}
In the above projection if the point $x$ is in the interior of $X$ then the projection is not different from the original vector field, i.e., $\Pi_X(x,F(t,x)) = F(t,x)$. On the other hand if the point $x$ is in the border of $X$, then the projection is just the component of the vector field that is tangential to the set $X$ at the point $x$. Let's consider for instance the case where the set $X$ is a box in $R^n$. Let $X = [a_1,b_1] \times ... \times [a_n,b_n]$ where $a_1 .. a_n$ and $b_1 ... b_n$ are real numbers. Then for each component of the vector field we have that
\begin{equation}
\Pi_X\left(x,F(t,x)\right)_i=\left\{ \begin{array}{l} 0 \quad \mbox{if} \quad x_i = a_i \quad \mbox{and} \quad F(t,x)_i < 0, \\
0 \quad \mbox{if} \quad x_i = b_i \quad \mbox{and} \quad F(t,x)_i > 0 ,\\
F(t,x)_i \quad \mbox{otherwise}.
\end{array}
\right.
\end{equation}
Therefore, when the projection is included, the proposed controller takes the form of the following projected dynamical system:
\begin{equation}\label{eqn_gradient_controller}
\dot{x} = \Pi_X\left(x,-\varepsilon f_{0,x}(t,x)\right),
\end{equation}
where $\varepsilon>0$ is the gain of the controller. Before stating the first theorem we need a Lemma concerning the relation between the original vector field and the projected vector field. This lemma is used in the proofs of theorems \ref{theo:first_theo}, \ref{theo:not_opti} and \ref{theo:opti}. 
\begin{lemma}\label{lemma:big_lemma}
Let $X$ be a convex set and $x_0 \in X$ and $x \in X$. Then
\begin{equation}\label{eqn_big_lemma}
(x_0-x)^T \Pi_X(x_0,v) \leq (x_0-x)^T v.
\end{equation}
\end{lemma}
\begin{proof}
See Apendix \ref{ap_big_lema_proof}.
\end{proof}

%%%%%%%%%%%%%%%%%%%%%%%%%%%%%%%%%%%%%%%%%%%%%%%%%%%%%%%%%%%%%%%%%%%%%%%%%%%
%%%   M   A   I   N       M   A   T   T   E   R   %%%%%%%%%%%%%%%%%%%%%%%%%
%%%%%%%%%%%%%%%%%%%%%%%%%%%%%%%%%%%%%%%%%%%%%%%%%%%%%%%%%%%%%%%%%%%%%%%%%%%
%
Let's define an Energy function $V_\bbarx:\reals^n \to \reals$ as
\begin{equation}\label{eqn_Lyapunov_opti}
V_\bbarx(x) =    \frac{1}{2} (x-\bbarx)^T(x-\bbarx).
\end{equation}
Where $\bbarx \in X \subset \reals^n$ is an arbitrary fixed action. We are now in conditions to present the first theorem, which states that the solution of the gradient controller defined in \eqref{eqn_gradient_controller} is a strongly optimal trajectory, i.e., with bounded regret for all $T$.

%%%%%%%%%%%%%%%%%%%%%%%%%%%%%%%%%%%%%%%%%%%%%%%%%%%%%%%%%%%%%%%%%%%%%%%%%%%
%%%   T   H   E   O   R   E   M   %%%%%%%%%%%%%%%%%%%%%%%%%%%%%%%%%%%%%%%%%
%%%%%%%%%%%%%%%%%%%%%%%%%%%%%%%%%%%%%%%%%%%%%%%%%%%%%%%%%%%%%%%%%%%%%%%%%%%
%
\begin{theorem}\label{theo:first_theo} Let $f_0: \reals\times X \to \reals$ be cost function satisfying assumptions 1 and 2, with  $X \subseteq \reals^n$ convex. The solution $x(t)$ of the online projected gradient controller in \eqref{eqn_gradient_controller} is strongly optimal in the sense of Definition \ref{def_0ptimality_and_viability}. In particular, the regret $\ccalR_T$ can be bounded by 
\begin{equation}\label{eqn_theo_first_theo}
   \ccalR_{T} \leq V_{x^*}\left(x(0)\right) / \varepsilon, \quad \text{for all\ } T\\
\end{equation}
where $V_\bbarx$ is the Energy function in \eqref{eqn_Lyapunov_opti}.\end{theorem}

%%%%%%%%%%%%%%%%%%%%%%%%%%%%%%%%%%%%%%%%%%%%%%%%%%%%%%%%%%%%%%%%%%%%%%%%%%%
%%%   P   R   O   O   F   %%%%%%%%%%%%%%%%%%%%%%%%%%%%%%%%%%%%%%%%%%%%%%%%%
%%%%%%%%%%%%%%%%%%%%%%%%%%%%%%%%%%%%%%%%%%%%%%%%%%%%%%%%%%%%%%%%%%%%%%%%%%%
%
\begin{proof}
Consider an action trajectory $x(t)$, an arbitrary given action $\bbarx \in X$, and the corresponding energy function $V_\bbarx(x(t))$ as per \eqref{eqn_Lyapunov_opti}. The derivative $\dot V_\bbarx(x(t))$ of the energy function with respect to time is then given by
\begin{equation}\label{eqn_theo_opti_pf_21}
   \dot{V}_\bbarx(x(t)) = (x(t) - \bbarx)^T\dot{x}(t).
\end{equation}
If the trajectory $x(t)$ follows from the online projected gradient dynamical system in \eqref{eqn_gradient_controller} we can substitute the trajectory derivative $\dot x$ by the vector field value and reduce \eqref{eqn_theo_opti_pf_21} to
\begin{equation}\label{eqn_theo_opti_pf_22}
   \dot{V}_\bbarx(x(t)) = (x(t) - \bbarx)^T \Pi_X \left(x(t),-\varepsilon f_{0,x}(t,x(t))\right).
\end{equation}
Use now the result in Lemma \ref{lemma:big_lemma} with $v=-\varepsilon f_{0,x}(t,x(t))$ to remove the projection operator from \eqref{eqn_theo_opti_pf_22} and write
\begin{equation}\label{eqn_theo_opti_pf_23}
   \dot{V}_\bbarx(x(t)) \leq -\varepsilon (x(t)-\bbarx)^Tf_{0,x}(t,x(t)).
\end{equation}
Using the defining equation of a subgradient \eqref{eqn_def_subgradient}, we can upper bound the inner product $-(x(t)-\bbarx )^T f_{0,x}(t,x(t))$ by the difference $f_0(t,\bbarx) - f_0(t,x(t))$ and transform \eqref{eqn_theo_opti_pf_23} into
\begin{equation}
   \dot{V}_\bbarx(x(t))\leq \varepsilon\left(f_0(t,\bbarx) - f_0(t,x(t))\right).
\end{equation}
Rearranging terms in the preceding inequality and integrating over time yields 
\begin{equation}\label{eqn_theo1basic}
   \int_0^T f_0(t,x(t)) \,dt - \int_0^T f_0(t,\bbarx) \,dt  
      \leq -\frac{1}{\varepsilon}\int_0^T \dot{V}_\bbarx(x(t))\,dt .
\end{equation}
Since the primitive of $\dot{V}_\bbarx(x(t))$ is $V_\bbarx(x(t))$ we can evaluate the integral on the right hand side of \eqref{eqn_theo1basic} and further use the fact that $V_\bbarx (x) \geq 0$ for all $x\in\reals^n$ to conclude that
\begin{equation}\label{eqn_theo1_aux}
    -\int_0^T \dot{V}_\bbarx(x(t)) dt 
        \ =   \  V_\bbarx(x(0)) - V_\bbarx (x(T))
        \ \leq\  V_\bbarx\left(x(0)\right) .
\end{equation}
Combining the bounds in \eqref{eqn_theo1basic} and \eqref{eqn_theo1_aux} we have that
\begin{equation}\label{eqn_theo1almost}
   \int_0^T f_0(t,x(t)) \,dt-\int_0^T f_0(t,\bbarx) \,dt  
       \leq  V_\bbarx(x(0)) / \varepsilon .
\end{equation}
Since the above inequality holds for an arbitrary point $\bbarx \in \reals^n$ it holds for $\bbarx=x^*$ in particular. When making $\bbarx=x^*$ in \eqref{eqn_theo1almost} the left hand side reduces to the regret $\ccalR_T$ associated with the trajectory $x(t)$ [cf. \eqref{eqn_continuous_regret}] and in the right hand side we have $V_\bbarx(x(0))/ \varepsilon = V_{x^*}(x(0))/ \varepsilon$. Eq. \eqref{eqn_theo_first_theo} follows because  \eqref{eqn_theo1almost} is true for all times $T$. This implies that the trajectory is strongly optimal according to \eqref{eqn_def_0ptimality} in Definition \ref{def_0ptimality_and_viability}. \end{proof}

%%%%%%%%%%%%%%%%%%%%%%%%%%%%%%%%%%%%%%%%%%%%%%%%%%%%%%%%%%%%%%%%%%%%%%%%%%%
%%%   M   A   I   N       M   A   T   T   E   R   %%%%%%%%%%%%%%%%%%%%%%%%%
%%%%%%%%%%%%%%%%%%%%%%%%%%%%%%%%%%%%%%%%%%%%%%%%%%%%%%%%%%%%%%%%%%%%%%%%%%%
%
The strong optimality of the online projected gradient controller in \eqref{eqn_gradient_controller} that we claim in Theorem \ref{theo:first_theo} is not a straightforward generalization of the optimality of gradient controllers in constant convex potentials. The functions $f_0$ are allowed to change arbitrarily over time and are not observed until after the cost $f_0(t,x(t))$ has been incurred.

Since the initial value of the Energy function $V_{x^*}(x(0))$ is the square of the distance between $x(0)$ and $x^*$, the bound on the regret in \eqref{eqn_theo_first_theo} shows that the closer we start to the optimal point the smaller the accumulated cost is. Likewise, the larger the controller gain $\varepsilon$, the smaller the bound on the regret is. Theoretically, we can make this bound arbitrarily small. This is not possible in practice because larger $\varepsilon$ entails trajectories with larger derivatives which cannot be implemented in systems with physical constraints. In the example in Section \ref{sec_shepherd_problem} the derivatives of the state $x(t)$ control the speed and acceleration of the shepherd. The physical limits of these quantities along with an upper bound on the cost gradient $f_{0,x}(t,x)$ can be used to estimate the largest allowable gain $\varepsilon$. 

Another observation regarding the bound on the regret is that it does not depend on the function that we are minimizing --except for the location of the point $x^*$. For instance by scaling a function the bound on the regret is kept constant if the same gain $\varepsilon$ can be selected. This is not surprising since a scaling in the function implies a bigger cost but it also entails a larger action derivative, which allows to track better changes on the function. However, if a bound on the maximum allowed gain exists then the regret bound cannot be invariant to scalings.
%%%%%%%%%%%%%%%%%%%%%%%%%%%%%%%%%%%%%%%%%%%%%%%%%%%%%%%%%%%%%%%%%%%%%%%%%%%
%%%   R   E   M   A   R   K   %%%%%%%%%%%%%%%%%%%%%%%%%%%%%%%%%%%%%%%%%%%%%
%%%%%%%%%%%%%%%%%%%%%%%%%%%%%%%%%%%%%%%%%%%%%%%%%%%%%%%%%%%%%%%%%%%%%%%%%%%
%
\begin{remark}\normalfont
In discrete time systems where $t$ is a natural variable and the integrals in \eqref{eqn_continuous_regret} are replaced by sums, online gradient descent algorithms are used to reduce regret; see e.g. \cite{Zinkevich03,hazan2007logarithmic}. The online gradient controller in \eqref{eqn_gradient_controller} is a direct generalization of online gradient descent to continuous time. This similarity notwithstanding, the result in Theorem \ref{theo:first_theo} is stronger than the corresponding bound on the regret in discrete time which states a sublinear growth at a rate not faster than $\sqrt{T}$ if the cost function is convex \cite{Zinkevich03}, and $\log{T}$ if the cost function is strictly convex \cite{hazan2007logarithmic}. The key where this difference lies is in the fact that discrete time algorithms have to "pay" to switch from the action at time $t$ to the action at time $t+1$. In the proofs of  \cite{Zinkevich03,hazan2007logarithmic} a term related to the norm square of the gradient is present in the upper bound on the regret while in continuous time this term is absent. The bound on the norm of the gradient is related to the selecting a different action. As in the case of fictitious plays that lead to no regret in the continuous time but not in discrete time (see e.g.\cite{viossat2013no,hart2001general,young1993evolution}) the bounds on the regret in continuous time are tighter than in discrete time for online gradient descent. 
\end{remark}

%!TEX root = 00_feasibility_continuous_time.tex

%%%%%%%%%%%%%%%%%%%%%%%%%%%%%%%%%%%%%%%%%%%%%%%%%%%%%%%%%%%%%%%%%%%%%%%%%%%
%%%   S   E   C   T   I   O   N   %%%%%%%%%%%%%%%%%%%%%%%%%%%%%%%%%%%%%%%%%
%%%%%%%%%%%%%%%%%%%%%%%%%%%%%%%%%%%%%%%%%%%%%%%%%%%%%%%%%%%%%%%%%%%%%%%%%%%
%
\section{Saddle point algorithm}
\label{sec:main}

Given an environment $f(t,x)$ and an objective function $f_0(t,x)$ verifying assumptions \ref{as:integrability} and \ref{as:convexity} we set our attention towards two different problems: design a controller whose solution is a strongly feasible trajectory and a controller whose solution is a feasible and strongly optimal trajectory. As already noted, when the environment is known beforehand the problem of finding such trajectories is a constrained convex optimization problem, which we can solve using the saddle point algorithm of Arrow and Hurwicz \cite{arrow_hurwicz}. Following this idea, let $\lambda \in \Lambda =\reals^m_+$, be a multiplier and define the time-varying Lagrangian associated with the online problem as 
\begin{equation}\label{eqn_lagrangian}
\mathcal{L}(t,x,\lambda) = f_0(t,x)+\lambda^Tf(t,x). 
\end{equation}
Saddle point methods rely on the fact that for a constrained convex optimization problem, a pair is a primal-dual optimal solution if and only if the pair is a saddle point of the Lagrangian associated with the problem; see e.g. \cite{boyd2004convex}. The main idea of the algorithm is then to generate trajectories that descend in the opposite direction of the gradient of the Lagrangian with respect to $x$ and that ascend in the direction of the gradient with respect to $\lambda$. 

Since the Lagrangian is differentiable with respect to $\lambda$, we denote by $\mathcal{L}_{\lambda}(t,x,\lambda)=f(t,x)$ the derivative of the Lagrangian with respect to $\lambda$. On the other hand, since the functions $f_0(\cdot,x)$ and $f(\cdot,x)$ are convex, the Lagrangian is also convex with respect to $x$. Thus, its subgradient with respect to $x$ always exist, let us denote it by $\mathcal{L}_x(t,x,\lambda)$. Let $\varepsilon$ be the gain of the controller, then following the ideas in \cite{arrow_hurwicz} we define a controller that descends in the direction of the subgradient with respect to the action $x$
\begin{align}\label{eqn_action_descent}
\dot{x} &\ =\ \Pi_X \left( x,- \varepsilon \mathcal{L}_x(t,x,\lambda) \right) \nonumber\\
        &\ =\ \Pi_X \left(x,-\varepsilon(f_{0,x}(t,x)+ f_x(t,x)\lambda) \right),
\end{align}
 and that ascends in the direction of the subgradient with respect to the multiplier $\lambda$
\begin{equation}\label{eqn_multiplier_ascent}
\dot{\lambda} = \Pi_{\Lambda} \left( \lambda,\varepsilon \mathcal{L}_{\lambda}(t,x,\lambda) \right)
= \Pi_{\Lambda} \left(\lambda, \varepsilon  f(t,x) \right).
\end{equation}
The projection over the set $X$ in \eqref{eqn_action_descent} is done to assure that the trajectory is always in the set of possible actions. The operator $\Pi_{\Lambda}(\lambda,f)$ is a projected dynamical system in the sense of Definition \ref{def_projected_dynamical_system} over the set $\Lambda$. This projection is done to assure that $\lambda(t) \in \reals^m_+$ for all times $t \in [0,T]$. An important observation regarding \eqref{eqn_action_descent} and \eqref{eqn_multiplier_ascent} is that the environment is observed locally in space and causally in time. The values of the environment constraints and its subgradients are observed at the current trajectory position $x(t)$ and the values of $f(t,x(t))$ and $f_x(t,x(t))$ affect the derivatives of $x(t)$ and $\lambda(t)$ only. Notice that if the environment function satisfies $f(t,x) \equiv 0$ we recover the algorithm defined in \eqref{eqn_gradient_controller} as a particular case of the saddle point controller. 

A block diagram for the controller in \eqref{eqn_action_descent} - \eqref{eqn_multiplier_ascent} is shown in Figure \ref{fig_block_diagram}. The controller operates in an environment to which it inputs at time $t$ an action $x(t)$ that results in a penalty $f(t,x(t))$ and cost $f_0(t,x(t))$. The value of these functions and their subgradients $f_x(t,x(t))$ and $f_{0,x}(t,x(t))$ are observed and fed to the multiplier and action feedback loops. The action feedback loop behaves like a weighted gradient descent controller. We move in the direction given by a linear combination of the the gradient of the objective function $f_{0,x}(t,x(t))$ and the constraint subgradients $f_{i,x}(t,x(t))$ weighted by their corresponding multipliers $\lambda_i(t)$. Intuitively, this pushes $x(t)$ towards satisfying the constraints and to the minimum of the objective function in the set where constraints are satisfied. However, the question remains of how much weight to give to each constraint. This is the task of the multiplier feedback loop. When constraint $i$ is violated we have $f_{i}(t,x(t))>0$. This pushes the multiplier $\lambda_i(t)$ up, thereby increasing the force $\lambda_i(t)f_{i,x}(t,x(t))$ pushing $x(t)$ towards satisfying the constraint. If the constraint is satisfied, we have $f_{i}(t,x(t))<0$, the multiplier $\lambda_i(t)$ being decreased, and the corresponding force decreasing. The more that constraint $i$ is violated, the faster we increase the multiplier, and the more we increase the force that pushes $x(t)$ towards satisfying $f_{i}(t,x(t))<0$. If the constraint is satisfied, the force is decreased and may eventually vanish altogether if we reach the point of making $\lambda_i(t)=0$.

%%%%%%%%%%%%%%%%%%%%%%%%%%%%%%%%%%%%%%%%%%%%%%%%%%%%%%%%%%%%%%%%%%%%%%%%%%%
%%%   F   I   G   U   R   E   %%%%%%%%%%%%%%%%%%%%%%%%%%%%%%%%%%%%%%%%%%%%%
%%%%%%%%%%%%%%%%%%%%%%%%%%%%%%%%%%%%%%%%%%%%%%%%%%%%%%%%%%%%%%%%%%%%%%%%%%%
%
\begin{figure}\centering
\resizebox{9cm}{7cm}{
%!TEX root = ../feasibility_continous_time.tex

\def \thisplotscale {1}
\def \unit {\thisplotscale cm}

\tikzstyle{block} = [draw, 
                     rectangle,
                     minimum height = 1*\unit, 
                     minimum width  = 4*\unit,
                     inner sep=3pt]

\tikzstyle{int block} = [draw, 
                     rectangle,
                     minimum height = 1*\unit, 
                     minimum width  = 1*\unit,
                     inner sep=3pt]

\begin{tikzpicture}[-stealth, shorten >=0, font=\small]

   % Shaded areas
   \path (-4.5,2.2) 
          node 
         [anchor = north west,
          fill = red!10, opacity = 0.5,
          rectangle, 
          minimum height = 3.3*\unit, 
          minimum width  = 11.5*\unit] (action descent) {};
   \path (action descent.north west) ++ (0,-0.25)
         node [right]{ \footnotesize Gradient descent on actions};

   \path (-4.5,-6.7) node 
         [anchor = south west,
          fill = blue!10, opacity = 0.5,
          rectangle, 
          minimum height = 3.2*\unit, 
          minimum width  = 11.5*\unit] (dual ascent)  {};
   \path (dual ascent.south west) ++ (0,+0.25)
         node [right]{ \footnotesize Gradient ascent on multipliers};

   \path (-4.5,-3.25) node 
         [anchor = south west,
          fill = green!10, opacity = 0.5,
          rectangle, 
          minimum height = 1.88*\unit, 
          minimum width  = 11.5*\unit] (environment) {};
   \path (environment.south east) ++ (0,+0.25)
         node [left]{ \footnotesize Environment};
         
   % Blocks for gradients
   \path          node [block] (gradx)  
                               {$\Pi_X\Big(x(t),-\varepsilon \left[f_{0,x}(t,x(t))+f_x(t,x(t))\lambda(t)\right]\Big)$}                            
        ++ (0,-4.5) node [block] (gradlambda) 
                               {$\Pi_\Lambda\Big(\lambda(t),\varepsilon f(t,x(t))\Big)$};
   % Blocks for integrals
   \path (gradx) ++ (5, 0)  node [int block] (intx)      {$\int$};
   \path (gradlambda-|intx) node [int block] (intlambda) {$\int$};

   % Blocks for environment
   \path (gradx) ++ (0, -2.15)  node [block] (environment)  {$f(t,x(t))$, $f_x(t,x(t))$, $f_{0,x}(t,x(t))$};

   % Arrows for direct flows
   \path (gradx) edge [above] node {$\dot x(t)$} (intx);    
   \path (gradlambda) edge [above] node {$\dot \lambda(t)$} (intlambda);    
   \path (intx) edge [above, pos=0.8] node {$x(t)$} ++ (2.5,0);    

   % Arrow for feedback flow from output x to input x
   \path[draw] (intx.east) ++ (0.5,0) -- ++(0, 1.5) -- ++ (-10,0) -- ++ (0,-1.25) -- ++ (0.77,0); 
   \path[draw](intx.east) ++ (0.5,0) --++(0,-2.15)--++(environment);   

   % Arrow for feedback flow from output lambda to input lambda 
   \path[draw] (intlambda.east) 
               -- ++ (0.5,0) node [above] {$\lambda(t)$} 
               -- ++(0,-1.5) -- ++ (-10.015,0) -- ++ (0,1.5) 
               -- ++ (2,0);    

   % Arrow for feedback flow from output lambda to input x 
   \path[draw] (gradlambda.west) ++(-2,0) -- ++ (0,4.25) -- ++(0.78,0);    

   % Arrows for environmental observations
   \path (environment) edge node [right] {} (gradx);    
   \path (environment) edge node [right] {} (gradlambda);

\end{tikzpicture}
}
\caption{Block diagram of the saddle point controller. Once that action $x(t)$ is selected at time $t$, we measure the corresponding values of $f(t,x(t))$, $f_x(t,x(t))$ and $f_{0,x}(t,x(t))$. This information is fed to the two feedback loops. The action loop defines the descent direction by computing weighted averages of the subgradients $f_x(t,x(t))$ and $f_{0,x}(t,x(t))$. The multiplier loop uses $f(t,x(t))$ to update the corresponding weights.} 
\label{fig_block_diagram}
\end{figure}
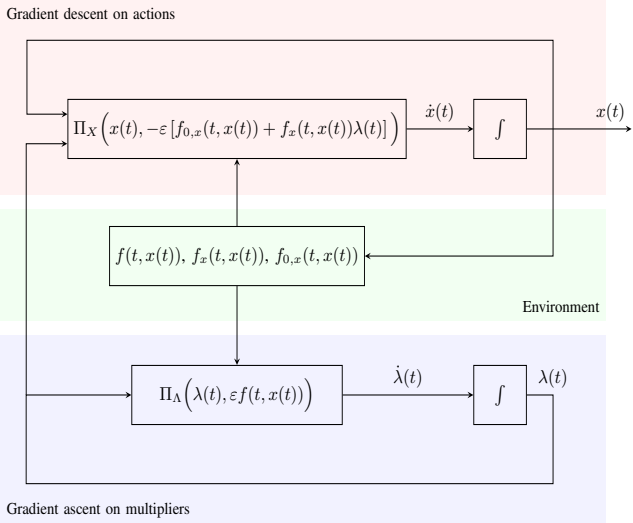

%%%%%%%%%%%%%%%%%%%%%%%%%%%%%%%%%%%%%%%%%%%%%%%%%%%%%%%%%%%%%%%%%%%%%%%%%%%
%%%   SECTION   %%%%%%%%%%%%%%%%%%%%%%%%%%%%%%%%%%%%%%%%%%%%%
%%%%%%%%%%%%%%%%%%%%%%%%%%%%%%%%%%%%%%%%%%%%%%%%%%%%%%%%%%%%%%%%%%%%%%%%%%%
%
\subsection {Strongly feasible trajectories}\label{subsec:non_opti}

We begin by studying the saddle point controller defined by \eqref{eqn_action_descent} and \eqref{eqn_multiplier_ascent} in a problem in which optimality is {\it not} taken into account, i.e., $f_0(t,x) \equiv 0$. In this case the action descent equation of the controller \eqref{eqn_action_descent} takes the form:
\begin{equation}\label{eqn_non_opti}
 \dot{x} = \Pi_X \left( x,- \varepsilon \mathcal{L}_x(t,x,\lambda) \right) = \Pi_X \left(x,-\varepsilon f_x(t,x)\lambda \right), 
\end{equation}
while the multiplier ascent equation \eqref{eqn_multiplier_ascent} remains unchanged. The bounds to be derived for the fit ensure that the trajectories $x(t)$ are strongly feasible in the sense of Definition \ref{def_0ptimality_and_viability}. To state the result consider an arbitrary fixed action $\bar{x} \in X$ and an arbitrary multiplier $\bar{\lambda} \in \Lambda$ and define the energy function
\begin{equation}\label{eqn_lyapunov}
   V_{\bbarx,\bar{\lambda}}(x,\lambda) 
       = \frac{1}{2}\left( ||x-\bbarx||^2+||\lambda -\bar{\lambda}||^2\right).
\end{equation}
We can then bound fit in terms of the initial value $V_{\bbarx,\bar{\lambda}}(x(0),\lambda(0))$ of the energy function for properly chosen $\bbarx$ and $\bar{\lambda}$ as we formally state next.
\begin{theorem}\label{theo:not_opti}
Let $f: \reals \times X \to \reals^m$, satisfying assumptions \ref{as:integrability} and \ref{as:convexity},  where $X \subseteq \reals^n$ is a convex set. If the environment is viable, then the solution $x(t)$ of the dynamical system defined by \eqref{eqn_non_opti} and \eqref{eqn_multiplier_ascent} is strongly feasible for all $T>0$. Specifically, the fit is bounded by
\begin{align}\label{eqn_fit_bound_non_opti}
   \mathcal{F}_{T,i} \leq \min_{x^\dagger \in X^\dagger} \frac{1}{\varepsilon} V_{x^\dagger, e_{i}}(x(0),\lambda(0)),
\end{align}
where $e_i$ with $i=1..m$ form the canonical base of $\reals^m$. 
\end{theorem}
\begin{proof}
Consider action trajectories $x(t)$ and multiplier trajectories $\lambda(t)$ and the  corresponding energy function $V_{\bbarx,\bar{\lambda}}(x(t),\lambda(t))$ in \eqref{eqn_lyapunov} for arbitrary given action $\bbarx \in X$ and multiplier $\bar{\lambda}\in \Lambda$. The derivative $\dot V_{\bbarx,\bar{\lambda}}(x(t),\lambda(t))$ of the energy with respect to time is then given by
\begin{equation}\label{eqn_theo_survival_pf_10}
   \dot{V}_{\bar{x},\bar{\lambda}} (x(t),\lambda(t))
      = (x(t) - \bbarx)^T\dot{x}(t) + (\lambda(t) -\bar{\lambda})^T\dot{\lambda}(t).
\end{equation}
Substitute the action and multiplier derivatives by their corresponding values given in \eqref{eqn_non_opti} and \eqref{eqn_multiplier_ascent} to reduce \eqref{eqn_theo_survival_pf_10} to
%If the trajectories $x(t)$ and $\lambda(t)$ follow from the saddle point dynamical system given by \eqref{eqn_non_opti} and \eqref{eqn_multiplier_ascent} we can substitute the action and multiplier derivatives by their corresponding values and reduce\eqref{eqn_theo_survival_pf_10} to
%
\begin{align}\label{eqn_theo_survival_pf_11}
\dot{V}_{\bbarx,\bar{\lambda}}(x(t),\lambda(t)) = 
 &(x(t) - \bbarx)^T  \Pi_X \left( x, - \varepsilon f_x(t,x(t))\lambda(t)\right)    \nonumber \\ &+(\lambda(t)-\bar{\lambda})^T \Pi_{\Lambda} \left(\lambda, \varepsilon f(t,x(t)) \right).
\end{align}
Then, using the result of Lemma \ref{lemma:big_lemma} for both $X$ and $\Lambda$, the following inequality holds:
\begin{align}\label{eqn_theo_survival_pf_11}
\dot{V}_{\bbarx,\bar{\lambda}}(x(t),\lambda(t)) &\leq 
\varepsilon (\bbarx-x(t))^T f_x(t,x(t))\lambda(t) \nonumber \\
&+\varepsilon(\lambda(t)-\bar{\lambda})^T f(t,x(t)). 
\end{align}
Notice that $f(t,x) \lambda(t)$ is a convex function with respect to the action, therefore we can upper bound the inner product $(\bar{x} - x(t))^Tf_x(t,x(t)) \lambda(t)$ by the quantity $f(t,\bar{x})^T\lambda(t) - f(t,x(t))^T \lambda(t)$ and transform \eqref{eqn_theo_survival_pf_11} into
\begin{align}\label{eqn_theo_survival_pf_12}
\dot{V}_{\bbarx,\bar{\lambda}}(x(t),\lambda(t))&\leq 
\varepsilon \left(f(t,\bbarx)-f(t,x(t))\right)^T\lambda(t) \nonumber \\
&+\varepsilon(\lambda(t)-\bar{\lambda})^T f(t,x(t)). 
\end{align}
Further note that in the above equation the second and the third term are opposite. Thus, it reduces to
\begin{equation}
\dot{V}_{\bbarx,\bar{\lambda}}(x(t),\lambda(t))\leq \varepsilon\left[\lambda^T(t)f(t,\bbarx) - \bar{\lambda}^T f(t,x(t))\right]. 
\end{equation}
Rewriting the above expression and then integrating both sides with respect to time from $t = 0 $ to $t =T$ we obtain
\begin{equation}\label{eqn_inter}
\begin{split}
\varepsilon \int_0^T \bigl(\bar{\lambda}^T f(t,x(t)) - \lambda^T(t) &f(t,\bbarx) \bigr)dt \\ &
 \leq - \int_0^T \dot{V}_{\bbarx,\bar{\lambda}}(x(t),\lambda(t)) dt. 
\end{split}
\end{equation}
Integrating the right side of the above equation we obtain
\begin{align}
-\int_0^T \dot{V}_{\bbarx,\bar{\lambda}}(x(t),&\lambda(t))dt  \\\nonumber &
= V_{\bbarx,\bar{\lambda}}(x(0),\lambda(0))-V_{\bbarx,\bar{\lambda}}(x(T),\lambda(T)).
\end{align}
Then using the fact that $V_{\bbarx,\bar{\lambda}}(x(t)),\lambda(t)) \geq 0$ for all $t$, yields
\begin{equation}\label{eqn_inequality_chain}
-\int_0^T \dot{V}_{\bbarx,\bar{\lambda}}(x(t),\lambda(t))dt\leq 
V_{\bbarx,\bar{\lambda}}\left(x(0),\lambda(0)\right). 
\end{equation}
Then, combining \eqref{eqn_inter} and \eqref{eqn_inequality_chain}, we have that
\begin{equation}
\int_0^T \bar{\lambda}^T f(t,x(t)) - \lambda^T(t) f(t,\bbarx) dt \leq \left( V_{x^\dagger,\bar{\lambda}}(x(0),\lambda(0))\right) / \varepsilon.
\label{eqn_final}
\end{equation}
Since the environment is viable, there exist a fixed action $x^{\dagger}$ such that $f(t,x^{\dagger})\leq 0$ for all $t \geq 0$. Then choosing $\bbarx = x^{\dagger}$, since $\lambda(t)\geq 0$ for all $t$, we have that %$\lambda^T(t) f(t,x^\dagger) \, \leq 0$ for all $t\in[0,T]$.
\begin{equation}
\lambda^T(t) f(t,x^\dagger) \, \leq 0  \; \forall t\in[0,T]. 
\end{equation}
Therefore the left hand side of \eqref{eqn_final} can be lower bounded by
\begin{equation}
\bar{\lambda}^T\int_0^Tf(t,x(t)) dt \leq \left(V_{x^\dagger,\bar{\lambda}}(x(0),\lambda(0)\right)/ \varepsilon.
\end{equation}
Choosing $\bar{\lambda} = e_i$ where $e_i$ is the $i$th element of the canonical base of $\reals^m$, we have that for all $i=1..m$:
\begin{equation}
\int_0^T f_i(t,x(t)) dt \leq \left( V_{x^\dagger,e_i}(x(0),\lambda(0)) \right) / \varepsilon.
\end{equation} 
Notice that since the above inequality holds for any $x^\dagger\in X^\dagger$ it is also true for the particular $x^\dagger$ that minimizes the right hand side. The left hand side of the above inequality is the $i$th component of the fit. Thus, since the $m$ components of the fit of the trajectory generated by the saddle point algorithm are bounded for all $T$, the trajectory is strongly feasible with the specific upper bound stated in \eqref{eqn_fit_bound_non_opti}.
\end{proof}
Theorem \ref{theo:not_opti} assures that if an environment is viable for an agent that selects actions over a set $X$, the solution of the dynamical system given by \eqref{eqn_non_opti} and \eqref{eqn_multiplier_ascent} is a trajectory $x(t)$ that is strongly feasible in the sense of Definition \ref{def_0ptimality_and_viability}. This result is not trivial, since the function $f$ that defines the environment is observed causally and can change arbitrarily over time. In particular, the agent could be faced with an adversarial environment that changes the function $f$ in a way that makes the value of $f(t,x(t))$ larger. The caveat is that the choice of the function $f$ must respect the viability condition that there exists a feasible action $x^\dagger$ such that $f(t,x^\dagger) \leq 0$ for all $t\in[0,T]$. This restriction still leaves significant leeway for strategic behavior. E.g., in the shepherd problem of Section \ref{sec_shepherd_problem} we can allow for strategic sheep that observe the shepherd's movement and respond by separating as much as possible. The strategic action of the sheep are restricted by the condition that the environment remains viable, which in this case reduces to the not so stringent condition that the sheep stay in a ball of radius $2r$ if all $r_i=r$.

Since the initial value of the energy function $V_{x^\dagger,e_i}(x(0),\lambda(0))$ is the square of the distance between $x(0)$ and $x^\dagger$ added to a term that depends on the distance between the initial multiplier and $e_i$, the bound on the fit in \eqref{eqn_fit_bound_non_opti} shows that the closer we start to the feasible set the smaller the accumulated constraint violation becomes. Likewise, the larger the gain $\varepsilon$, the smaller the bound on the fit is. 
As in section \ref{sec:continuous_regret} we observe that increasing $\varepsilon$ can make the bound on the fit arbitrarily small, yet for the same reasons discussed in that section this can't be done. 

Further notice that for the saddle point controller defined by \eqref{eqn_non_opti} and \eqref{eqn_multiplier_ascent}  the action derivatives are proportional not only to the gain $\varepsilon$ but to the value of the multiplier $\lambda$. Thus, to select gains that are compatible with the system's physical constraints we need to determine upper bounds in the multiplier values $\lambda(t)$. An upper bound follows as a consequence of Theorem \ref{theo:not_opti} as we state in the following corollary.

\begin{corollary}\label{coro_bounded_multipliers}Given the controller defined by \eqref{eqn_non_opti} and \eqref{eqn_multiplier_ascent} and assuming the same hypothesis of Theorem \ref{theo:not_opti}, if the set of actions $X$ is bounded in norm by $R$, then the multipliers $\lambda$ are bounded for all times by
\begin{equation}\label{eqn_multiplier_bound}
0 \leq \lambda_i(t) \leq \left(4R^2+1\right), \ \mbox{for all} \ i=1,\ldots,m.
\end{equation}
\end{corollary}
\begin{proof}
First of all notice that according to \eqref{eqn_multiplier_ascent} a projection over the positive orthant is performed for the multiplier update. Therefore, for each component of the multiplier we have that $\lambda_i(t) \geq 0 $ for all $t\in[0,T]$. On the other hand, since the trajectory of the multipliers is defined by $\dot{\lambda}(t) = \Pi_\Lambda(\lambda(t),\varepsilon f(t,x(t))$, while $\lambda(t) >0$ we have that $\dot{\lambda}(t) = \varepsilon f(t,x(t))$. Let $t_0$ be the first time instant for which $\lambda_i (t) > 0$ for a given $i\in\{1,2,..,m\}$, i.e.,
\begin{equation}
t_0 = \inf \left\{ {t\in [0,T]}, \lambda_i(t) >0 \right\}. 
\end{equation}
In addition, let $T^*_0$ be the first time instant greater than $t_0$ where $\lambda_i(t) = 0$, if this time is larger than $T$ we set $T^*_0 = T$, formally this is
\begin{equation}\label{eqn_T_estrella}
T^*_0 = \max \left\{\inf \left\{ {t\in (t_0,T]}, \lambda_i(t) >0 \right\} , T\right\}. 
\end{equation}
Further define 
%
%\begin{equation}
$t_{s+1} = \inf \left\{ {t\in [T_s^*,T]}, \lambda_i(t) >0 \right\},$
%\end{equation}
%
and 
\begin{equation}\label{eqn_T_estrella}
T^*_s = \max \left\{\inf \left\{ {t\in (t_s,T]}, \lambda_i(t) >0 \right\} , T\right\}. 
\end{equation}
From the above definition it holds that in any time in the interval $(T^*_s, t_{s+1}]$, we have $\lambda_i(t)=0$. And therefore in those intervals the multipliers are bounded. Consider now $\tau \in (t_s,T^*_s]$. In this case it holds that
\begin{equation}
\int_{t_s}^\tau \dot{\lambda}_i(t)dt = \ \int_{t_s}^\tau \varepsilon f_i(t,x(t))dt.
\end{equation}
%\begin{align}
%\int_{t_s}^\tau \dot{\lambda}_i(t)dt 
% &\ =\ \int_{t_s}^\tau \Pi_{\lambda_i}(\lambda(t),\varepsilon f(t,x(t)))dt  \nonumber \\
% &\ =\ \int_{t_s}^\tau \varepsilon f_i(t,x(t))dt.
%\end{align}
%
Notice that the right hand side of the above equation is, proportional to the $i$th component of the fit restricted to the time interval $[t_0,\tau]$. In Theorem \ref{theo:not_opti} it was proved that the $i$th component of the fit is bounded for all time horizons by $V_{x^\dagger,e_i}(x(t_s),0)/\varepsilon$. In this particular case we have that
\begin{equation}
V_{x^\dagger,e_i}(x(t_s),0) = \frac{1}{2}\left((x(t_s)-x^\dagger)^2 + (0-e_i)^2\right),
\end{equation}
and since for any $x\in X$ we have that $\|x\| \leq R$, we conclude
\begin{equation}
V_{x^\dagger,e_i}(x(t_s),0) \leq \frac{1}{2}\left((2R)^2 +1^2\right).
\end{equation}
Therefore, for all $\tau \in (t_sT^*_s]$
%
%\begin{equation}\label{eqn_lambda_bound}
$\lambda_i(\tau)  \leq \frac{1}{2}\left(4R^2 +1^2\right)$.
%\end{equation}
%
This completes the proof that the multipliers are bounded.
\end{proof}
The bound in Corollary \ref{coro_bounded_multipliers} ensures that action derivatives $\dot x(t)$ remain bounded if the subgradients are. This means that action derivatives increase, at most, linearly with $\varepsilon$ and is not compounded by an arbitrary increase of the multipliers. 

The cumulative nature of the fit does not guarantee that the constraint violation is controlled. This is because time intervals of constraint violations can be compensated by time intervals where the constraints are negative. Thus, it is of interest to show that the saddle point controller archives bounded saturated fit for all time horizon. We formalize this result next.
\begin{corollary}\label{corollary_saturated_fit}
Let the hypothesis of Theorem \ref{theo:not_opti} hold. Let $\delta>0$ and let $\bar{\ccalF}_{T}$ be the saturated fit defined in \eqref{eqn_saturated_fit}. Then, the solution of the dynamical system \eqref{eqn_non_opti} and \eqref{eqn_multiplier_ascent} when $f(t,x)$ is replaced by $\bar{f}_{\delta}(t,x)) = \max\left\{f(t,x),-\delta \right\}$ archives a bounded saturated fit. Furthermore the bound is given by 
\begin{equation}
   \bar\ccalF_{T,i} \leq \min_{x^\dagger \in X^\dagger} \frac{1}{\varepsilon} V_{x^\dagger, e_{i}}(x(0),\lambda(0)),
   \end{equation}
   where $e_i$ with $i=1..m$ form the canonical base of $\reals^m$. 

\end{corollary}
\begin{proof}
Since $\bar{f}_{\delta}(t,x)$ is the pointwise maximum of two convex functions, it is a convex function itself. As a consequence of Theorem \ref{theo:not_opti} the fit for the environment $\bar{f}_{\delta}(t,x)$ satisfies
\begin{equation}
\int_0^T \bar{f}_\delta(t,x(t))\, dt \leq \min_{x^\dagger \in X^\dagger} \frac{1}{\varepsilon} V_{x^\dagger, e_{i}}(x(0),\lambda(0)).
\end{equation}  
The fact that the left hand side of the above equation corresponds to the saturated fit [c.f. \eqref{eqn_saturated_fit}] completes the proof. 
\end{proof}
The above result establishes that a trajectory that follows the saddle point dynamics for the environment defined by $\bar{f}_{\delta}(t,x)$ achieves bounded saturated fit. This means that it is possible to adapt the controller \eqref{eqn_non_opti} and \eqref{eqn_multiplier_ascent}, so that the fit is bounded while not alternating between periods of large under and over satisfaction of the constraints
 
%
%%%%%%%%%%%%%%%%%%%%%%%%%%%%%%%%%%%%%%%%%%%%%%%%%%%%%%%%%%%%%%%%%%%%%%%%%%%
%%%   S   E   C   T   I   O   N   %%%%%%%%%%%%%%%%%%%%%%%%%%%%%%%%%%%%%%%%%
%%%%%%%%%%%%%%%%%%%%%%%%%%%%%%%%%%%%%%%%%%%%%%%%%%%%%%%%%%%%%%%%%%%%%%%%%%%
%
\subsection{Strongly optimal feasible trajectories}\label{subsec:otpi}

This section presents bounds on the growth of the fit and the regret of the trajectories $x(t)$ that are solutions of the saddle point controller defined by \eqref{eqn_action_descent} and \eqref{eqn_multiplier_ascent}. These bounds ensure that the trajectory is feasible and strongly optimal in the sense of Definition \ref{def_0ptimality_and_viability}. To derive these bounds we need the following assumption regarding the objective function. 
%
%
%
%%%%%%%%%%%%%%%%%%%%%%%%%%%%%%%%%%%%%%%%%%%%%%%%%%%%%%%%%%%%%%%%%%%%%%%%%%%
%%%   A   S   S   U   M   P   T   I   O   N   %%%%%%%%%%%%%%%%%%%%%%%%%%%%%
%%%%%%%%%%%%%%%%%%%%%%%%%%%%%%%%%%%%%%%%%%%%%%%%%%%%%%%%%%%%%%%%%%%%%%%%%%%
%
\begin{assumption} \label{as:lower_bound}
There is a finite constant $K$ independent of the time horizon $T$ such that for all $t$ in the interval $[0,T]$. 
\begin{equation}\label{eqn_constant_for_lemma}
 K \geq f_0(t,x^*) -\min_{x\in X } f_0(t,x),
\end{equation} 
where $x^*$ is the solution of the offline problem \eqref{eqn_optimal_strategy}.
\end{assumption}

%%%%%%%%%%%%%%%%%%%%%%%%%%%%%%%%%%%%%%%%%%%%%%%%%%%%%%%%%%%%%%%%%%%%%%%%%%%
%%%   M   A   I   N       M   A   T   T   E   R   %%%%%%%%%%%%%%%%%%%%%%%%%
%%%%%%%%%%%%%%%%%%%%%%%%%%%%%%%%%%%%%%%%%%%%%%%%%%%%%%%%%%%%%%%%%%%%%%%%%%%
%
The existence of the bound in \eqref{eqn_constant_for_lemma} is a mild requirement. Since the function $f_0(t,x)$ is convex, for any time $t$ it is lower bounded if the action space is bounded, as is the case in most applications of practical interest. The only restriction imposed is that $\min_{x\in X } f_0(t,x)$ does not become progressively smaller with time so that a uniform bound $K$ holds for all times $t$. The bound can still hold if $X$ is not compact as long as the span of the functions $f_0(t,x)$ is not unbounded below. A consequence of Assumption \ref{as:lower_bound} is that the regret cannot {\it decrease} faster than a linear rate as we formally state in the following lemma.

%%%%%%%%%%%%%%%%%%%%%%%%%%%%%%%%%%%%%%%%%%%%%%%%%%%%%%%%%%%%%%%%%%%%%%%%%%%
%%%   L   E   M   M   A   %%%%%%%%%%%%%%%%%%%%%%%%%%%%%%%%%%%%%%%%%%%%%%%%%
%%%%%%%%%%%%%%%%%%%%%%%%%%%%%%%%%%%%%%%%%%%%%%%%%%%%%%%%%%%%%%%%%%%%%%%%%%%
%
\begin{lemma}\label{lemma:regret_lower_bound} 
Let $X \subset \reals^n$ be a convex set. If Assumption \ref{as:lower_bound} holds, then the regret defined in  \eqref{eqn_continuous_regret} is lower bounded by $-KT$ where $K$ is the constant defined in \eqref{eqn_constant_for_lemma}, i.e.,
\begin{equation}\label{eqn_lemma_regret_lower_bound} 
   \ccalR_T \geq -KT.
\end{equation} \end{lemma}
\begin{proof}
See Appendix \ref{ap_regret_lower_bound}.
\end{proof}

%%%%%%%%%%%%%%%%%%%%%%%%%%%%%%%%%%%%%%%%%%%%%%%%%%%%%%%%%%%%%%%%%%%%%%%%%%%
%%%   M   A   I   N       M   A   T   T   E   R   %%%%%%%%%%%%%%%%%%%%%%%%%
%%%%%%%%%%%%%%%%%%%%%%%%%%%%%%%%%%%%%%%%%%%%%%%%%%%%%%%%%%%%%%%%%%%%%%%%%%%
%
Observe that regret is a quantity that we want to make small and, therefore, having negative regret is a desirable outcome. The result in Lemma \ref{lemma:regret_lower_bound} puts a floor on how much we can succeed in making regret negative. Using the bound in \eqref{eqn_lemma_regret_lower_bound} and the definition of the energy function in \eqref{eqn_lyapunov} we can formalize bounds on the regret and the fit,  for an action trajectory $x(t)$ that follows the saddle point dynamics in \eqref{eqn_action_descent} and \eqref{eqn_multiplier_ascent}. 

%%%%%%%%%%%%%%%%%%%%%%%%%%%%%%%%%%%%%%%%%%%%%%%%%%%%%%%%%%%%%%%%%%%%%%%%%%%
%%%   T   H   E   O   R   E   M   %%%%%%%%%%%%%%%%%%%%%%%%%%%%%%%%%%%%%%%%%
%%%%%%%%%%%%%%%%%%%%%%%%%%%%%%%%%%%%%%%%%%%%%%%%%%%%%%%%%%%%%%%%%%%%%%%%%%%
%
\begin{theorem}\label{theo:opti}
Let $X \subset \reals^n$ be a compact convex set and let $f: \reals \times X \to \reals^m$ and $f_0: \reals \times X \to \reals$, be functions satisfying assumptions \ref{as:integrability}, \ref{as:convexity} and \ref{as:lower_bound}. If the environment is viable, then the solution of the system defined by \eqref{eqn_action_descent} and \eqref{eqn_multiplier_ascent} is a trajectory $x(t)$ that is feasible and strongly optimal for all time horizons $T>0$ if the gain $\varepsilon >1$. In particular, the fit is bounded by
\begin{equation}\label{eqn_penalty_bound}
\ccalF_{T,i} \leq \ccalO\left(\sqrt{KT},\varepsilon^0\right),\end{equation}
and the regret is bounded by
\begin{equation}\label{eqn_regret_upper_bound_full_problem}
\ccalR_T\leq \frac{1}{\varepsilon} V_{x*,0} \left(x(0),\lambda(0)\right),
\end{equation}
where $V_{\bbarx,\bar{\lambda}}(x,\lambda)$ is the energy function defined in \eqref{eqn_lyapunov}, $x^*$ is the solution to the problem \eqref{eqn_optimal_strategy} and $K$ is the constant defined in  \eqref{eqn_constant_for_lemma}. The notation $\ccalO\left(\varepsilon^0\right)$ refers to a function that is constant with respect to the gain $\varepsilon$.
\end{theorem}

%%%%%%%%%%%%%%%%%%%%%%%%%%%%%%%%%%%%%%%%%%%%%%%%%%%%%%%%%%%%%%%%%%%%%%%%%%%
%%%   P   R   O   O   F   %%%%%%%%%%%%%%%%%%%%%%%%%%%%%%%%%%%%%%%%%%%%%%%%%
%%%%%%%%%%%%%%%%%%%%%%%%%%%%%%%%%%%%%%%%%%%%%%%%%%%%%%%%%%%%%%%%%%%%%%%%%%%
%
\begin{proof}
See Appendix \ref{ap_theo_opti}
\end{proof}

%%%%%%%%%%%%%%%%%%%%%%%%%%%%%%%%%%%%%%%%%%%%%%%%%%%%%%%%%%%%%%%%%%%%%%%%%%%
%%%   M   A   I   N       M   A   T   T   E   R   %%%%%%%%%%%%%%%%%%%%%%%%%
%%%%%%%%%%%%%%%%%%%%%%%%%%%%%%%%%%%%%%%%%%%%%%%%%%%%%%%%%%%%%%%%%%%%%%%%%%%
%
Theorem \ref{theo:opti} assures that if the environment is viable for an agent selecting actions from a bounded set $X$, the solution of the saddle point dynamics defined in \eqref{eqn_action_descent}-\eqref{eqn_multiplier_ascent} is a trajectory that is feasible and strongly optimal. The bounds on the fit in theorems \ref{theo:not_opti} and \ref{theo:opti} prove a trade off between optimality and feasibility. If optimality of the trajectory is not of interest it is possible to get strongly feasible trajectories with fit that is bounded by a constant independent of the time horizon $T$ (cf. Theorem \ref{theo:not_opti}). When an optimality criterion is added to the problem, its satisfaction may come at the cost of a fit that may increase as $\sqrt{T}$. An important consequence of this difference is that even if we could set the gain $\varepsilon$ to be arbitrarily large, the bound on the fit cannot be made arbitrarily small. This bound would still grow as $\sqrt{KT}$. The result in Theorem \ref{theo:opti} also necessitates Assumption \ref{as:lower_bound} as opposed to Theorem \ref{theo:not_opti}. 

As in the cases of theorems \ref{theo:first_theo} and \ref{theo:not_opti} it is possible to have the environment and objective function selected strategically. Further note that, again, the initial value of the energy function used to bound regret is related with the square of the distance between the initial action and the optimal offline solution of problem \eqref{eqn_optimal_strategy}. It also follows from the proof that this distance is related to the bound on the fit. Thus, the closer we start from this action the tighter the bounds will be. We next show that similar results holds for the saddle point dynamics if we consider the notion of saturated fit in lieu of fit.  
\begin{corollary}\label{corollary_saturated_fit2}
Let the hypothesis of Theorem \ref{theo:opti} hold. Let $\delta>0$ and let $\bar{\ccalF}_{T}$ be the saturated fit defined in \eqref{eqn_saturated_fit}. Then, the solution of the dynamical system \eqref{eqn_action_descent} and \eqref{eqn_multiplier_ascent}, when $f(t,x)$ is replaced by $\bar{f}_{\delta}(t,x)) = \max\left\{f(t,x),-\delta \right\}$ achieves a regret satisfying \eqref{eqn_regret_upper_bound_full_problem} and saturated fit that is bounded by 
\begin{equation}
   \bar\ccalF_{T,i} \leq \ccalO\left(\sqrt{KT},\varepsilon^0\right).
   \end{equation}
%
%and the regret satisfies \eqref{eqn_regret_upper_bound_full_problem}.
\end{corollary}
\begin{proof}
Same as Corollary \ref{corollary_saturated_fit}.
\end{proof}
The above result establishes that a trajectory that follows the saddle point dynamics for the environment defined by $\bar{f}_{\delta}(t,x)$ achieves bounded saturated fit. This means that it is possible to adapt the controller \eqref{eqn_action_descent} and \eqref{eqn_multiplier_ascent}, so that the growth of the fit is controlled while not alternating between periods of large under and over satisfaction of the constraints.
In the next section we evaluate the performance of the saddle point controller, after a pertinent remark on the selection of the gain. 

%%%%%%%%%%%%%%%%%%%%%%%%%%%%%%%%%%%%%%%%%%%%%%%%%%%%%%%%%%%%%%%%%%%%%%%%%%%
%%%  R E M A R K   %%%%%%%%%%%%%%%%%%%%%%%%%%%%%%%%%%%%%%%%%
%%%%%%%%%%%%%%%%%%%%%%%%%%%%%%%%%%%%%%%%%%%%%%%%%%%%%%%%%%%%%%%%%%%%%%%%%%%
%
\begin{remark}[\bf{Gain depending on the Time Horizon}] If it were possible to select the gain as a function of the time horizon $T$, fit could be bounded by a constant that does not grow with $T$. Take \eqref{eqn_algun_numero} and choose $\bar{\lambda} = e_i T$, where $e_i$ is the $i$-th component of the canonical base of $\mathbb{R}^m$ we have that
\begin{equation}
T\int_0^Tf_i(t,x(t)) dt \leq \left( V_{x^*,Te_i} (x(0),\lambda(0))\right) / \varepsilon +KT.
\end{equation}
With this selection of $\bar{\lambda}$ the function $V_{x^*,Te_i} \left(x(0),\lambda(0))\right)$ grows like $T^2$. Dividing both sides of the above equation by $T$ we have that the $i$-th component of the fit is bounded by
\begin{equation}\label{eqn_rmk_variable_gain_constant_fit}
\ccalF_{T,i} \leq \ccalO(T)/\varepsilon +K.
\end{equation}
If the gain is set to have order $\ccalO(T)$, the right hand side of \eqref{eqn_rmk_variable_gain_constant_fit} becomes of order $\ccalO(T^0)$. This means that fit can be bounded by a constant that does not depend on $T$. 
\end{remark}
%
%%%%%%%%%%%%%%%%%%%%%%%%%%%%%%%%%%%%%%%%%%%%%%%%%%%%%%%%%%%%%%%%%%%%%%%%%%%
%%%   S   E   C   T   I   O   N   %%%%%%%%%%%%%%%%%%%%%%%%%%%%%%%%%%%%%%%%%
%%%%%%%%%%%%%%%%%%%%%%%%%%%%%%%%%%%%%%%%%%%%%%%%%%%%%%%%%%%%%%%%%%%%%%%%%%%
%
\section{Numerical experiments}\label{sec:examples}
We evaluate performance of the saddle point algorithm defined by \eqref{eqn_action_descent}-\eqref{eqn_multiplier_ascent} in the solution of the shepherd problem introduced in Section \ref{sec_shepherd_problem}. We determine sheep paths using a perturbed polynomial characterization akin to the one in \eqref{eqn_shepherd_position}. Specifically, letting $p_j(t)$ be elements of a polynomial basis, the path $ y_{i} (t) = [y_{i,1} (t), y_{i,2} (t)]^T$ of the $i$th sheep is given by 
\begin{equation}\label{eqn_sheep_position}
  y_{i,k} (t) =\sum_{j=0}^{n_i-1} y_{i,k,j} p_j(t) + w_{i,k}(t),
\end{equation}
where $k=1,2$ denotes different path components, $n_i$ the dimension of the base that parameterizes the path followed by sheep $i$, and $y_{i,k,j}$ represent the corresponding $n_i$ coefficients. The noise terms $w_{i,k}(t)$ are Gaussian white with zero mean, standard deviation $\sigma$ and independent across components and sheep. Their purpose is to obtain more erratic paths. 

To determine $y_{i,k,j}$ we make $w_{i,k}(t)=0$ in \eqref{eqn_sheep_position} and require all sheep to start at $ y_{i} (0) =[0,0]^T$ and finish at $y_{i} (T) =[1,1]^T$. A total of $L$ random points $\{\tdy_l\}_{l=1}^L$ are then drawn independently and uniformly at random in the unit box $[0,1]^2$. Sheep $i=1$ is required to pass through points $\tdy_l$ at times $lT/(L+1)$, i.e., $y_1(lT/(L+1))=\tdy_l$. For each of the other sheep $i\neq 1$ we draw $L$ random offsets $\{\Delta\tdy_{i,l}\}_{l=1}^L$ uniformly at random from the box $[-\Delta,\Delta]^2$ and require the $i$th sheep path to satisfy $y_i(lT/(L+1))=\tdy_l + \Delta\tdy_{i,l}$. Paths $y_i(t)$ are then chosen as those that minimize the path integral of the acceleration squared subject to the constraints of each path
\begin{alignat}{2}\label{eqn_quadratic program}
   y^*_{i}  
     = &\argmin && \int_{0}^T \|\ddot{y}_{i} (t)\|^2 dt,    \nonumber\\
       &\st     && y_{i} (0) =[0,0]^T, \quad
                   y_{i} (T) =[1,1]^T,                      \nonumber\\  
       &        && y_i(lT/(L+1))=\tdy_l + \Delta\tdy_{i,l} ,
\end{alignat}
where, by construction $\Delta\tdy_{1,l}=0$. The paths in \eqref{eqn_quadratic program} can be computed as solutions of a quadratic program \cite{DM:11}. Let $y_i^*(t)$ be the trajectory given by \eqref{eqn_sheep_position} when we set $y_{i,k,j} =y_{i,k,j}^*$. We obtain the paths $y_{i,k} (t)$ by adding $w_{i,k}(t)$ to $y^*_{i} (t)$.

In subsequent numerical experiments we consider $m=5$ sheep, a time horizon $T=1$, and set the proximity constraint in \eqref{eqn_sheep_environment} to $r_i=0.3$. We use the polynomial basis $p_j(t)=t^j$ in both, \eqref{eqn_shepherd_position} and \eqref{eqn_sheep_position}. The number of basis elements in both cases is set to $n=n_i=30$. To generate sheep paths we consider a total of $L=3$ randomly chosen intermediate points, set the variation parameter to $\Delta=0.1$, and the perturbation standard deviation to $\sigma=0.1$. These problem parameters are such that the environment is most likely viable in the sense of Definition \ref{def_viable_environment}. We check that this is true by solving the offline feasibility problem. If the environment is not viable a new one is drawn before proceeding to the implementation of \eqref{eqn_action_descent}-\eqref{eqn_multiplier_ascent}.

We emphasize that even if the path of the sheep is known to us, the information is not used by the controller. The latter is only fed information of the position of the sheep at the current time, which it uses to evaluate the environment functions $f_i(t,x)$ in \eqref{eqn_sheep_environment}, their gradients $f_{ix}(t,x)$ and the gradient of $f_0(t,x)$. In the first problem considered $f_0(t,x)$ is identically zero, in the second takes the form of \eqref{eqn_black_sheep} and in the last problem the form of \eqref{eqn_minimum acceleration}. Since the agent is dynamicless, there are not physical constraints on the derivatives of the system, therefore the gain $\varepsilon$ in \eqref{eqn_action_descent}-\eqref{eqn_multiplier_ascent} can be set to have any value.

%%%%%%%%%%%%%%%%%%%%%%%%%%%%%%%%%%%%%%%%%%%%%%%%%%%%%%%%%%%%%%%%%%%%%%%%%%%
%%%   S   E   C   T   I   O   N   %%%%%%%%%%%%%%%%%%%%%%%%%%%%%%%%%%%%%%%%%
%%%%%%%%%%%%%%%%%%%%%%%%%%%%%%%%%%%%%%%%%%%%%%%%%%%%%%%%%%%%%%%%%%%%%%%%%%%
%
\subsection{Strongly feasible trajectories}\label{sec_pure_feasibility}
We consider a problem without optimality criterion in which case \eqref{eqn_action_descent}-\eqref{eqn_multiplier_ascent} simplifies to \eqref{eqn_non_opti}-\eqref{eqn_multiplier_ascent} and the strong feasibility result in Theorem \ref{theo:not_opti} applies. The system's behavior is illustrated in Figure \ref{fig:trajectory} when the gain is set to $\varepsilon = 50$. In this problem the average and maximal speed of the sheep is $5.1km/h$ and $14.8km/h$ respectively while for the shepherd these are $6.1km/h$ and $18.3 km/h$ for the selected gain. This speeds are in in the range of reasonable velocities for this particular problem. A qualitative examination of the sheep and shepherd paths shows that the shepherd succeeds in following the herd. A more quantitative evaluation is presented in Figure \ref{fig_relation_violation_multipliers} where we plot the instantaneous constraint violation $f_i(t,x(t))$ with respect to each sheep for the trajectories $x(t)$. Observe the oscillatory behavior that has the constraint violations $f_i(t,x(t))$ hovering at around $f_i(t,x(t))=0$. When the constraints are violated, i.e., when $f_i(t,x(t))>0$, the saddle point controller drives the shepherd towards a position that makes him stay within $r_i$ of all sheep. When a constraint is satisfied we have $f_i(t,x(t))<0$. This drives the multiplier $\lambda_i(t)$ towards 0 and removes the force that pushes the shepherd towards the sheep (c.f. Figure \ref{fig_relation_violation_multipliers}). The absence of this force makes the constraint violation grow and eventually surpass the maximum tolerance $f_i(t,x(t))=0$. At this point the multipliers start to grow and, as a consequence, to push the shepherd back towards proximity with the sheep. 

The behavior observed in Figure \ref{fig_relation_violation_multipliers} does not contradict the result in Theorem \ref{theo:not_opti} which gives us a guarantee on fit, not on instantaneous constraint violations. The components of the fit are shown in Figure \ref{fig:constraint_violation} where we see that they are indeed bounded. Thus, the trajectory is feasible in the sense of Definition \ref{def_0ptimality_and_viability}, even if the instantaneous problem's constraints are being violated at specific time instances. Further note that the fit is not only bounded but actually becomes negative. This is a consequence of the relatively large gain $\varepsilon=50$ which helps the shepherd to respond quickly to the sheep movements. The fit for a second experiment in which the gain is reduced to $\varepsilon=5$ is shown in Figure \ref{fig:violation_5eps}. In this case the fit stabilizes at a positive value. This behavior is expected because reducing $\varepsilon$ decreases the speed with which the shepherd can adapt to changes in the sheep paths. More to the point, the bound on the fit in Theorem \ref{theo:not_opti} is inversely proportional to the gain $\varepsilon$. The paths and instantaneous constraints violations for $\varepsilon=5$ are not shown but they are  qualitatively similar to the ones shown for $\varepsilon=50$ in figures \ref{fig:trajectory} and \ref{fig_relation_violation_multipliers}.
%
%
%%%%%%%%%%%%%%%%%%%%%%%%%%%%%%%%%%%%%%%%%%%%%%%%%%%%%%%%%%%%%%%%%%%%%%%%%%%
%%%   F   I   G   U   R   E   %%%%%%%%%%%%%%%%%%%%%%%%%%%%%%%%%%%%%%%%%%%%%
%%%%%%%%%%%%%%%%%%%%%%%%%%%%%%%%%%%%%%%%%%%%%%%%%%%%%%%%%%%%%%%%%%%%%%%%%%%
%
%
\begin{figure}\centering
\includegraphics[width=\linewidth, height=0.62\linewidth]{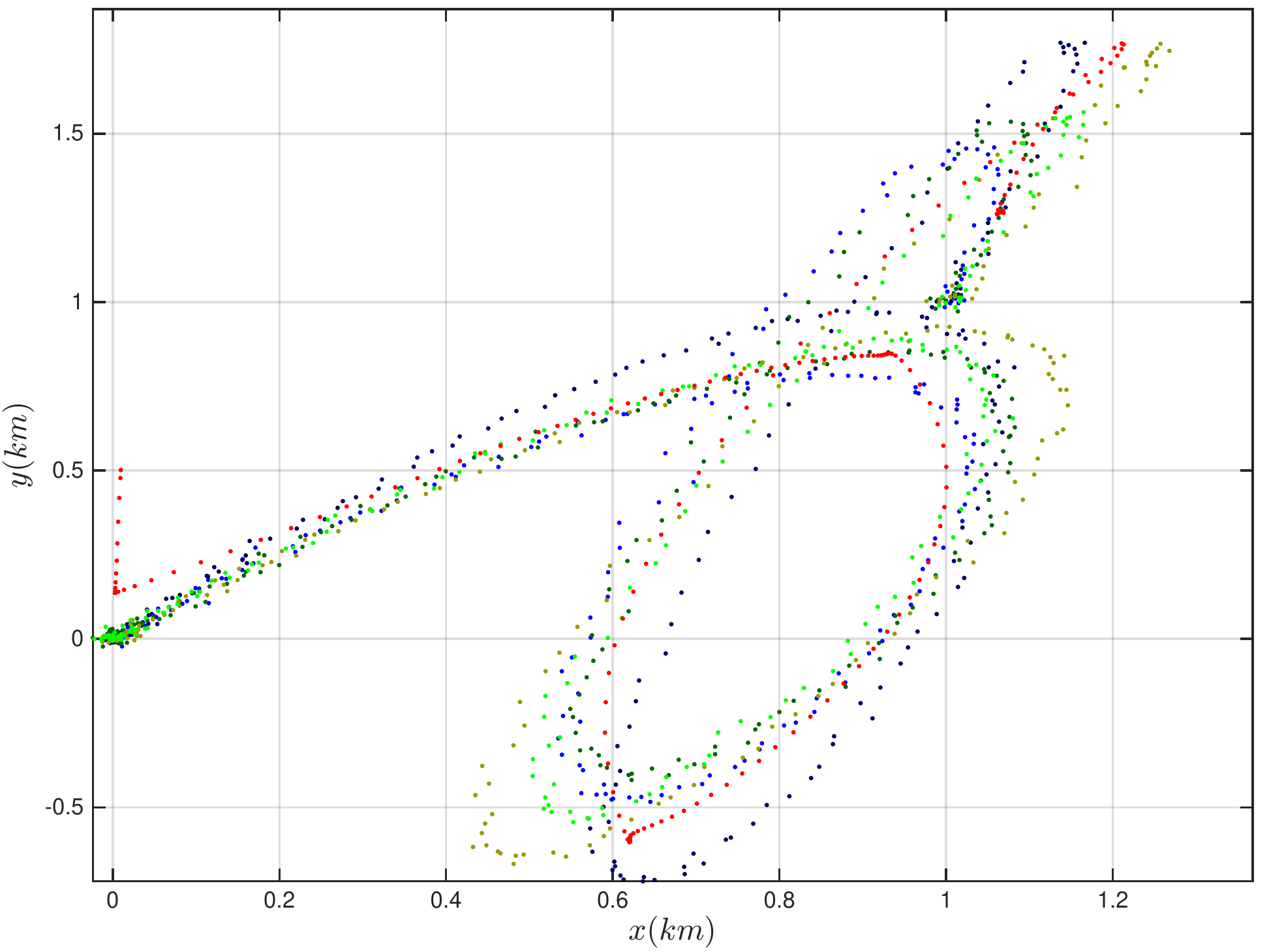} 
\caption{Path of the sheep and the shepherd for the feasibility-only problem (Section \ref{sec_pure_feasibility}) when the gain of the saddle point controller is set to be $\varepsilon =50$. The shepherd succeed in following the herd since its path -- in red -- is close to the path of all sheep.}
\label{fig:trajectory}\end{figure}
%
%%%%%%%%%%%%%%%%%%%%%%%%%%%%%%%%%%%%%%%%%%%%%%%%%%%%%%%%%%%%%%%%%%%%%%%%%%%
%%%   F   I   G   U   R   E   %%%%%%%%%%%%%%%%%%%%%%%%%%%%%%%%%%%%%%%%%%%%%
%%%%%%%%%%%%%%%%%%%%%%%%%%%%%%%%%%%%%%%%%%%%%%%%%%%%%%%%%%%%%%%%%%%%%%%%%%%
\begin{figure}
       \centering
        \begin{subfigure}[b]{\linewidth}
               \includegraphics[width=\linewidth, height=0.62\linewidth]{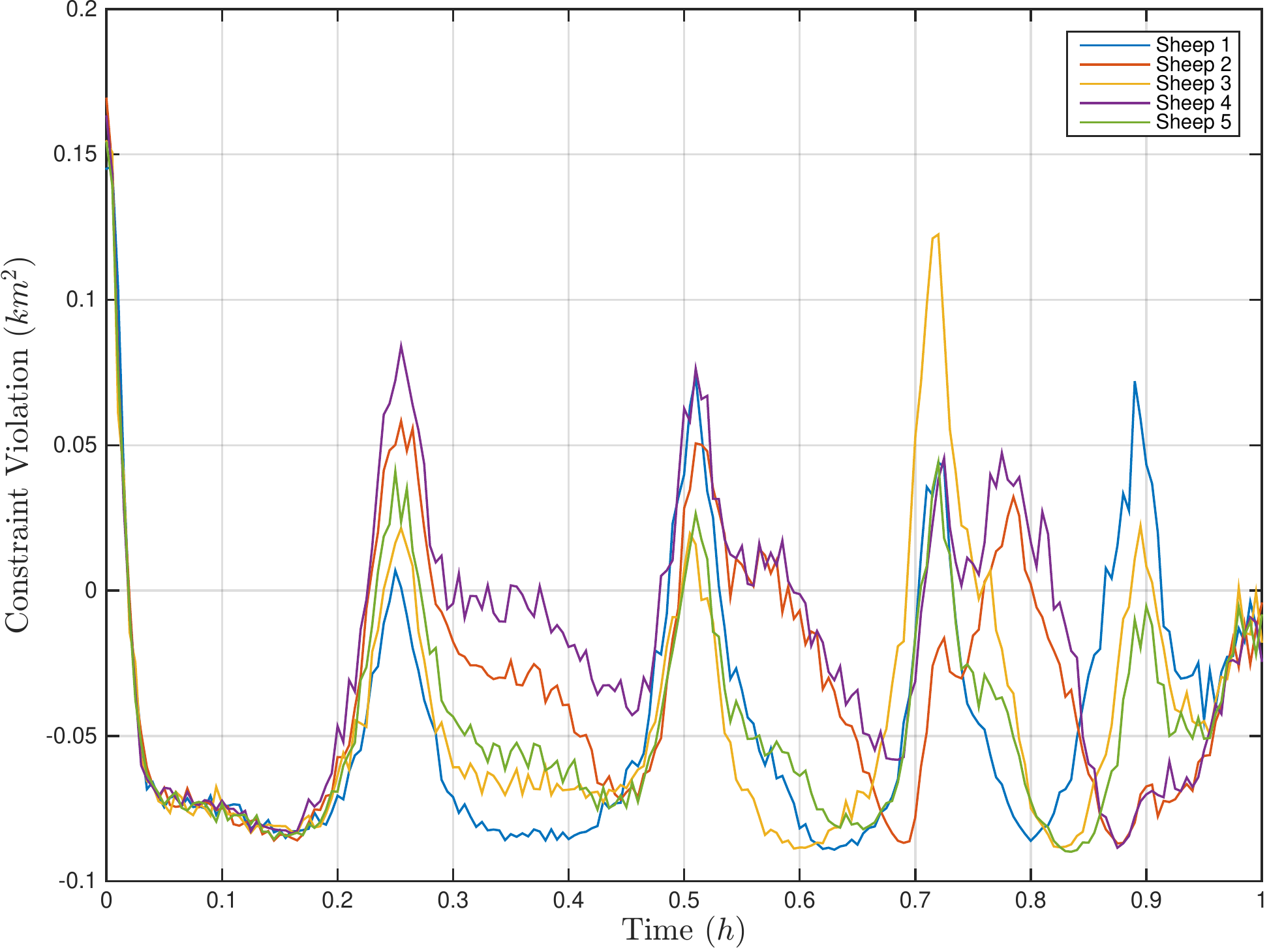}
               \caption{Instantaneous constraint value. } \label{fig:instant_violation}
        \end{subfigure}\par\vfill \bigskip
        \begin{subfigure}[b]{\linewidth}
                \includegraphics[width=\linewidth, height=0.62\linewidth]{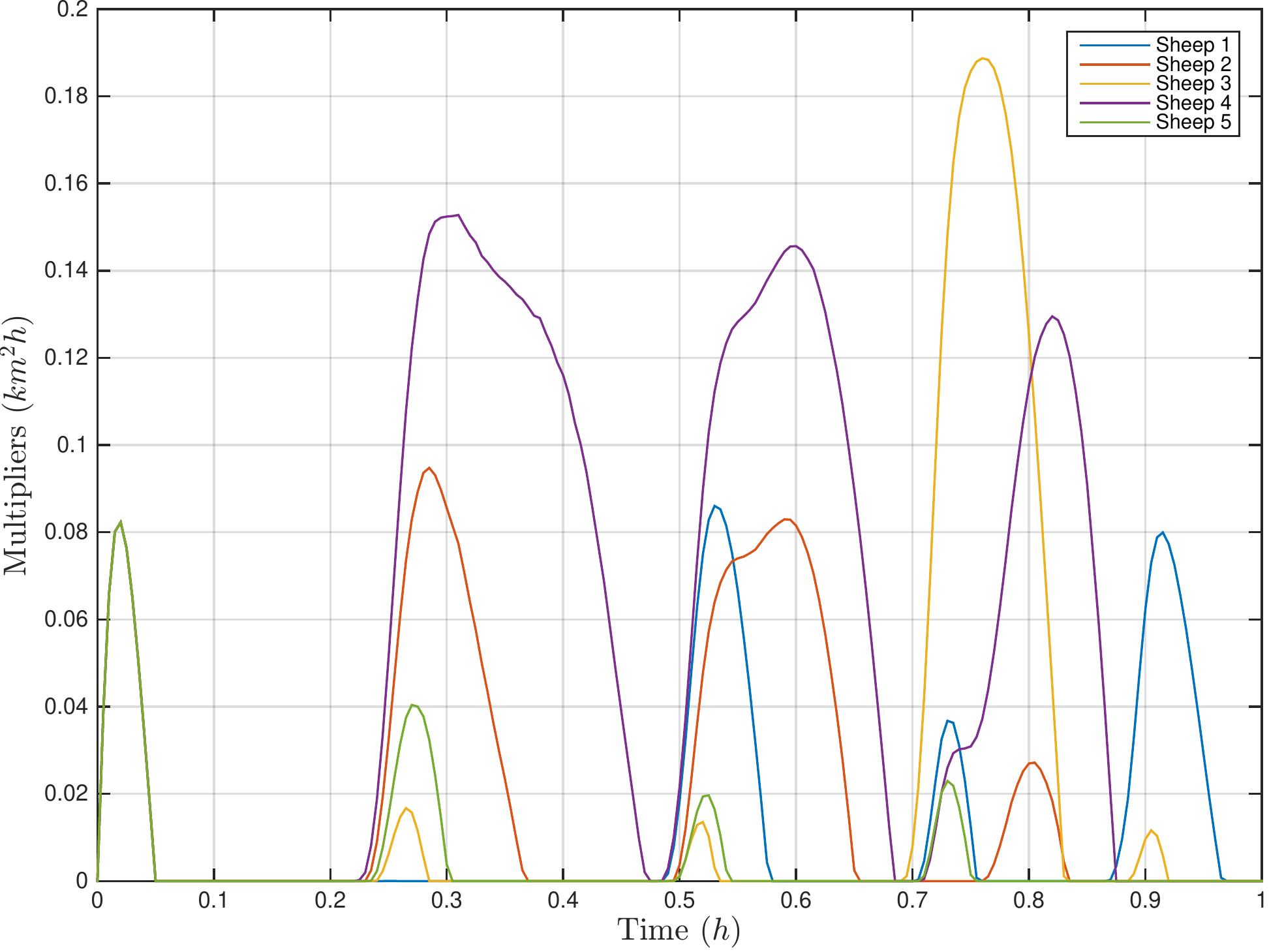}
\caption{Temporal evolution of the multipliers.} \label{fig_multipliers_feasibility}
        \end{subfigure}\bigskip
        \caption{Relationship between the instantaneous value of the constraints and their corresponding multipliers for the feasibility-only problem (Section \ref{sec_pure_feasibility}). At the times in which the value of a constraint is positive, its corresponding multiplier increases. When the value of the multipliers is large enough a decrease of the value of the constraint function is observed. Once the constraint function is negative the corresponding multiplier decreases until it reaches zero. }\label{fig_relation_violation_multipliers}
\end{figure}
%
%
%
%
%%%%%%%%%%%%%%%%%%%%%%%%%%%%%%%%%%%%%%%%%%%%%%%%%%%%%%%%%%%%%%%%%%%%%%%%%%%
%%%   F   I   G   U   R   E   %%%%%%%%%%%%%%%%%%%%%%%%%%%%%%%%%%%%%%%%%%%%%
%%%%%%%%%%%%%%%%%%%%%%%%%%%%%%%%%%%%%%%%%%%%%%%%%%%%%%%%%%%%%%%%%%%%%%%%%%%
%

\begin{figure*}
        \centering
        \begin{subfigure}[b]{0.5\linewidth}
                \includegraphics[width=\linewidth, height=0.62\linewidth]{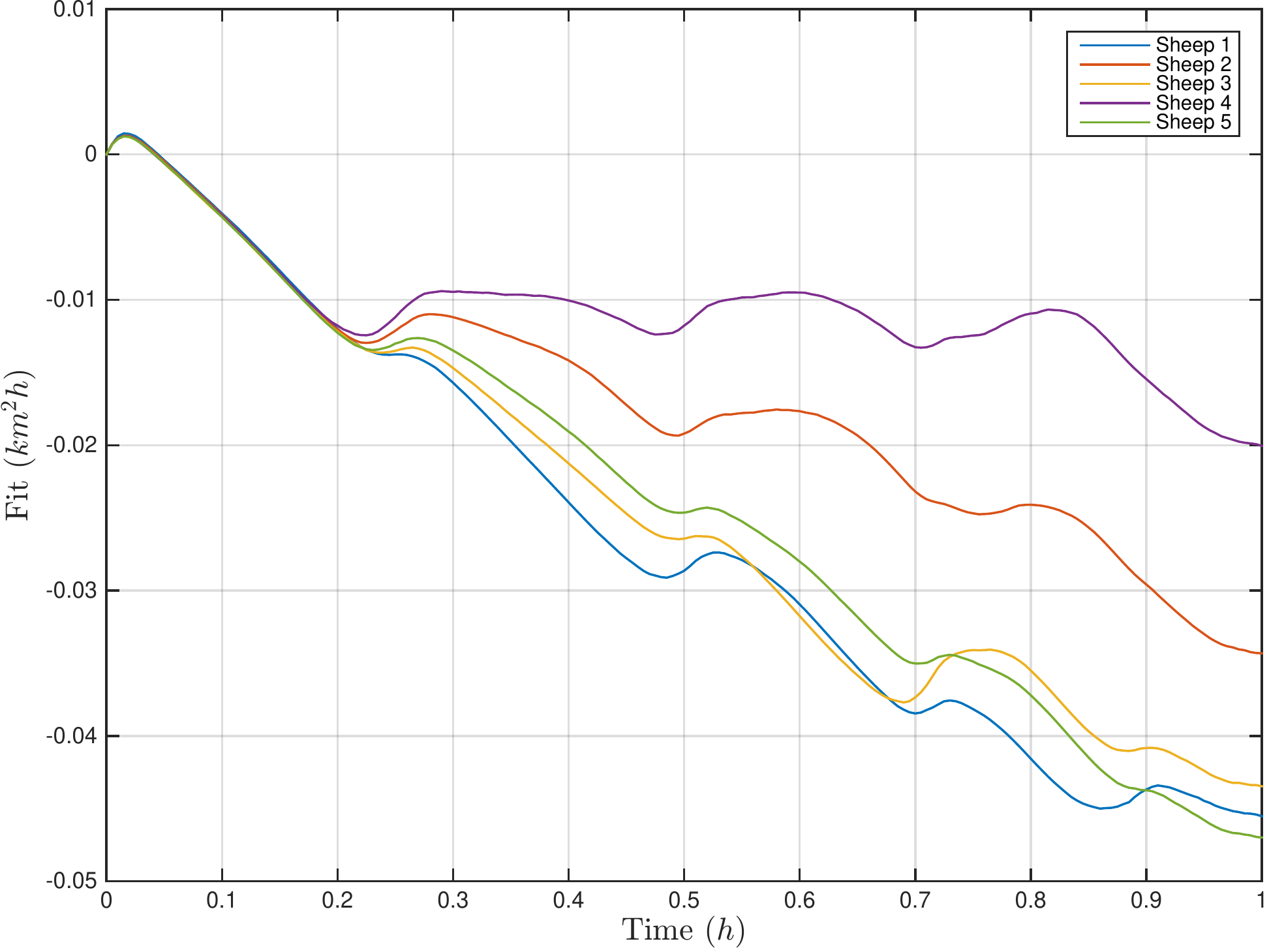}
                \caption{Experiment with gain $\varepsilon = 50$.}
                \label{fig:constraint_violation}
        \end{subfigure}%
        ~ %add desired spacing between images, e. g. ~, \quad, \qquad, \hfill etc.
          %(or a blank line to force the subfigure onto a new line)
        \begin{subfigure}[b]{0.5\linewidth}
                \includegraphics[width=\linewidth, height=0.62\linewidth]{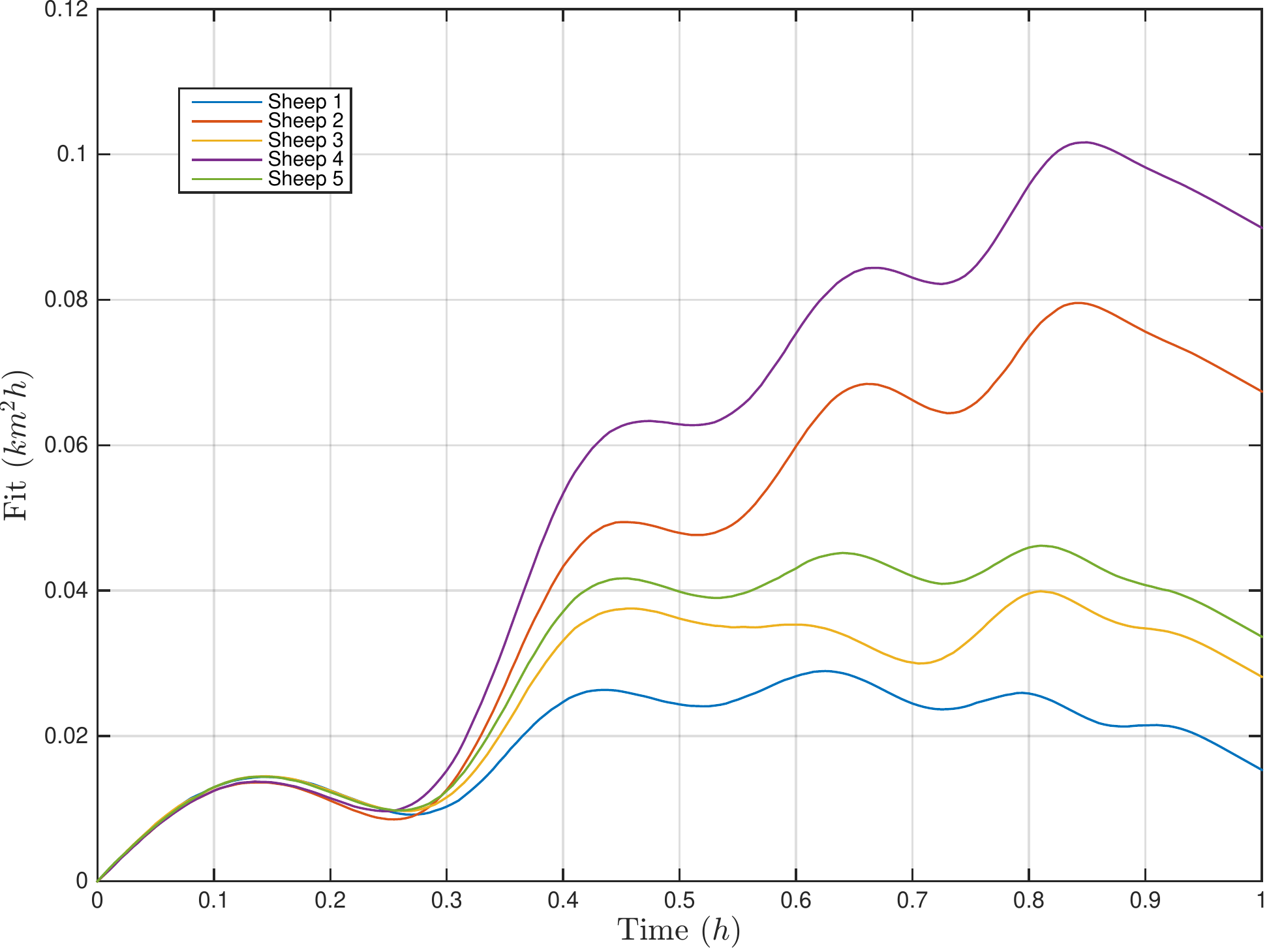}
                \caption{Experiment with gain $\varepsilon = 5$.}
                \label{fig:violation_5eps}
        \end{subfigure}
        ~ %add desired spacing between images, e. g. ~, \quad, \qquad, \hfill etc.
          %(or a blank line to force the subfigure onto a new line)
        \caption{Fit $\mathcal{F}_T$ for two different controller gains in the feasibility-only problem (Section \ref{sec_pure_feasibility}). Fit is bounded in both cases as predicted by Theorem \ref{theo:not_opti}. As is also predicted by Theorem \ref{theo:not_opti}, the larger the value of the gain $\varepsilon$ the smaller the bound on the fit of the shepherd's trajectory.}\label{fig_fit}
\end{figure*}
%
%
%
%
%
%%%%%%%%%%%%%%%%%%%%%%%%%%%%%%%%%%%%%%%%%%%%%%%%%%%%%%%%%%%%%%%%%%%%%%%%%%%%%%%%%%%
%%%% S U B S E C T I O N %%%%%%%%%%%%%%%%%%%%%%%%%%%%%%%%%%%%%%%%%
%%%%%%%%%%%%%%%%%%%%%%%%%%%%%%%%%%%%%%%%%%%%%%%%%%%%%%%%%%%%%%%%%%%%%%%%%%%%%%%%%%%
%
\subsection{Preferred sheep problem}\label{sec_preferred_sheep}
Besides satisfying the constraints in \eqref{eqn_sheep_environment}, the shepherd wishes to follow the first (black) sheep as close as possible. This translates into the optimality criterion \eqref{eqn_black_sheep}. Since the sheep trajectories are viable the hypotheses of Theorem \ref{theo:opti} hold. Thus, for a shepherd following the dynamics  \eqref{eqn_action_descent} and \eqref{eqn_multiplier_ascent}, the resulting trajectory is feasible and strongly optimal. 

Given that the trajectory is guaranteed to be feasible, we expect to have the fit bounded by a sublinear function of $T$. This does happen, as can be seen in the fit trajectories illustrated in Figure \ref{fig_fit_preferred_sheep} where a gain $\varepsilon =50$ is used. In fact, the fit does not grow and is bounded by a constant for all time horizons $T$. The trajectory is therefore not only feasible but strongly feasible. This does not contradict Theorem \ref{theo:opti} because strong feasibility implies feasibility. The reason why it's reasonable to see bounded fit here is that the objective function pushing the shepherd closer to the sheep is, in a sense, redundant with the constraints that push the shepherd to stay closer to all sheep. This redundancy can be also observed in the fact that the fit in this problem (c.f. Figure \ref{fig_fit_preferred_sheep}) is smaller than the fit in the problem of Section \ref{sec_pure_feasibility} (c.f. Figure \ref{fig:constraint_violation}). To explain why this may happen, focus on the value of the multipliers in Figure \ref{fig_multipliers_feasibility} between, e.g., times  $0.07\text{h} < t < 0.21\text{h}$. During this time the multipliers are equal to zero because all constraints are satisfied. As a consequence, the Lagrangian subgradient with respect to the action is identically zero in the time interval. In turn, this implies that the action is constant and no effort is made to reduce the value of the constraints. If the optimality criterion was present, the shepherd would be pushed towards the black sheep and fit would be further reduced.

The regret corresponding to the trajectory for this experiment with $\varepsilon =50$ is shown in Figure \ref{fig_regret_preferred_sheep}. Since the trajectory is strongly optimal as per Theorem \ref{theo:opti}, we expect regret to be bounded. This is the case in Figure \ref{fig_regret_preferred_sheep}
%
% where regret is actually negative for all times $t\in[0,T]$. Negative regret implies that the trajectory of the shepherd is incurring a total cost that is smaller than the one associated with the optimal solution. Notice that while the optimal fixed action minimizes the total cost as defined in \eqref{eqn_optimal_strategy} it does not minimize the objective at all times. Thus, by selecting different actions the shepherd can suffer smaller instantaneous losses than the ones associated with the optimal action. If this is the case, regret -- which is the integral of the difference between these two losses -- can be negative. 
%
The path of the shepherd is not shown for this experiment as it is qualitatively analogous to the one in Figure \ref{fig:trajectory} for the feasibility-only problem considered in Section \ref{sec_pure_feasibility}. 

%%%%%%%%%%%%%%%%%%%%%%%%%%%%%%%%%%%%%%%%%%%%%%%%%%%%%%%%%%%%%%%%%%%%%%%%%%%%%%%%%%%
%%%% S U B S E C T I O N %%%%%%%%%%%%%%%%%%%%%%%%%%%%%%%%%%%%%%%%%
%%%%%%%%%%%%%%%%%%%%%%%%%%%%%%%%%%%%%%%%%%%%%%%%%%%%%%%%%%%%%%%%%%%%%%%%%%%%%%%%%%%
%
\subsection{Minimum acceleration problem}\label{sec_minimum_acceleration}

We consider, an environment defined by the distances between the shepherd and the sheep given by \eqref{eqn_sheep_environment}, with the minimum acceleration objective defined in \eqref{eqn_minimum acceleration}. Since the construction of the target trajectories gives a viable environment we satisfy, again, the hypotheses of Theorem \ref{theo:opti}. Hence, for a shepherd following the dynamics given by \eqref{eqn_action_descent} and \eqref{eqn_multiplier_ascent}, the action trajectory is feasible and strongly optimal. In this section the gain of the controller is set to $\varepsilon = 50$. 

A feasible trajectory implies that the fit must be bounded by a function that grows sublinearly with the time horizon $T$. Notice that this is the case in Figure \ref{fig_fit_acceleration}. Periods of growth of the fit are observed, yet the presence of inflection points is an evidence of the growth being controlled. The fit in this problem is larger than the one in problem \ref{sec_preferred_sheep} (c.f figures \ref{fig_fit_preferred_sheep} and \ref{fig_fit_acceleration}). This result is predictable since the constraints and the objective function push the action in different directions. For instance, suppose that all constraints are satisfied and that the Lagrange multipliers are zero. Then, the subgradient of the Lagrangian is equal to the subgradient of the objective function. Hence the action will be modified trying to minimize the acceleration without taking the constraints (distance with the sheep) into account. Hence, pushing the action to the boundary of the feasible set. In this problem, this translates into the fact that the shepherd does not follow the sheep as closely as in the problems in sections  \ref{sec_pure_feasibility} and \ref{sec_preferred_sheep} (c.f Figure \ref{fig_trajectory_acceleration}).   

Since the trajectory is strongly optimal, we should observe a regret bounded by a constant. This is the case in Figure \ref{fig_regret_acceleration}, where in fact we observe negative regret for some time intervals. Negative regret implies that the trajectory of the shepherd is incurring a total cost that is smaller than the one associated with the optimal solution. Notice that while the optimal fixed action minimizes the total cost as defined in \eqref{eqn_optimal_strategy} it does not minimize the objective at all times. Thus, by selecting different actions the shepherd can suffer smaller instantaneous losses than the ones associated with the optimal fixed action. If this is the case, regret -- which is the integral of the difference between these two losses -- can be negative. 
%
%%%%%%%%%%%%%%%%%%%%%%%%%%%%%%%%%%%%%%%%%%%%%%%%%%%%%%%%%%%%%%%%%%%%%%%%%%%
%%%   F   I   G   U   R   E   %%%%%%%%%%%%%%%%%%%%%%%%%%%%%%%%%%%%%%%%%%%%%
%%%%%%%%%%%%%%%%%%%%%%%%%%%%%%%%%%%%%%%%%%%%%%%%%%%%%%%%%%%%%%%%%%%%%%%%%%%
%
\begin{figure}\centering
\includegraphics[width=\linewidth, height=0.62\linewidth]{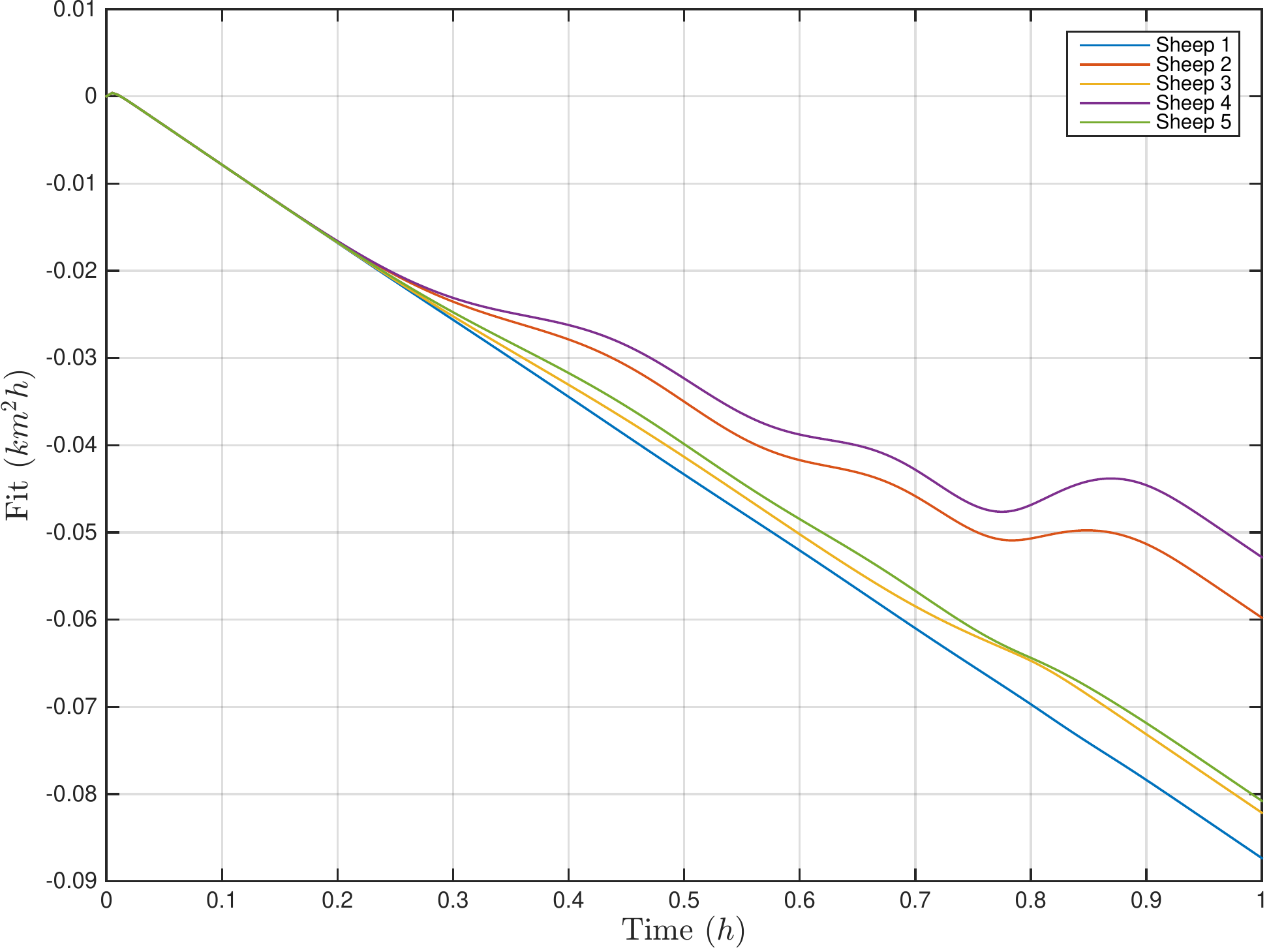}
\caption{Fit $\mathcal{F}_T$ for the preferred sheep problem (Section \ref{sec_preferred_sheep}) when the gain of the saddle point controller is set to be $\varepsilon=50$. As predicted by Theorem \ref{theo:opti} the trajectory is feasible since the fit is bounded, and, in fact, appears to be strongly feasible. Since the subgradient of the objective function is the same as the subgradient of the first constrain the fit is smaller than in the pure feasibility problem (c.f Figure \ref{fig_fit}). }
\label{fig_fit_preferred_sheep}\end{figure}

%%%%%%%%%%%%%%%%%%%%%%%%%%%%%%%%%%%%%%%%%%%%%%%%%%%%%%%%%%%%%%%%%%%%%%%%%%%
%%%   F   I   G   U   R   E   %%%%%%%%%%%%%%%%%%%%%%%%%%%%%%%%%%%%%%%%%%%%%
%%%%%%%%%%%%%%%%%%%%%%%%%%%%%%%%%%%%%%%%%%%%%%%%%%%%%%%%%%%%%%%%%%%%%%%%%%%
%
\begin{figure}\centering
\includegraphics[width=\linewidth, height=0.62\linewidth]{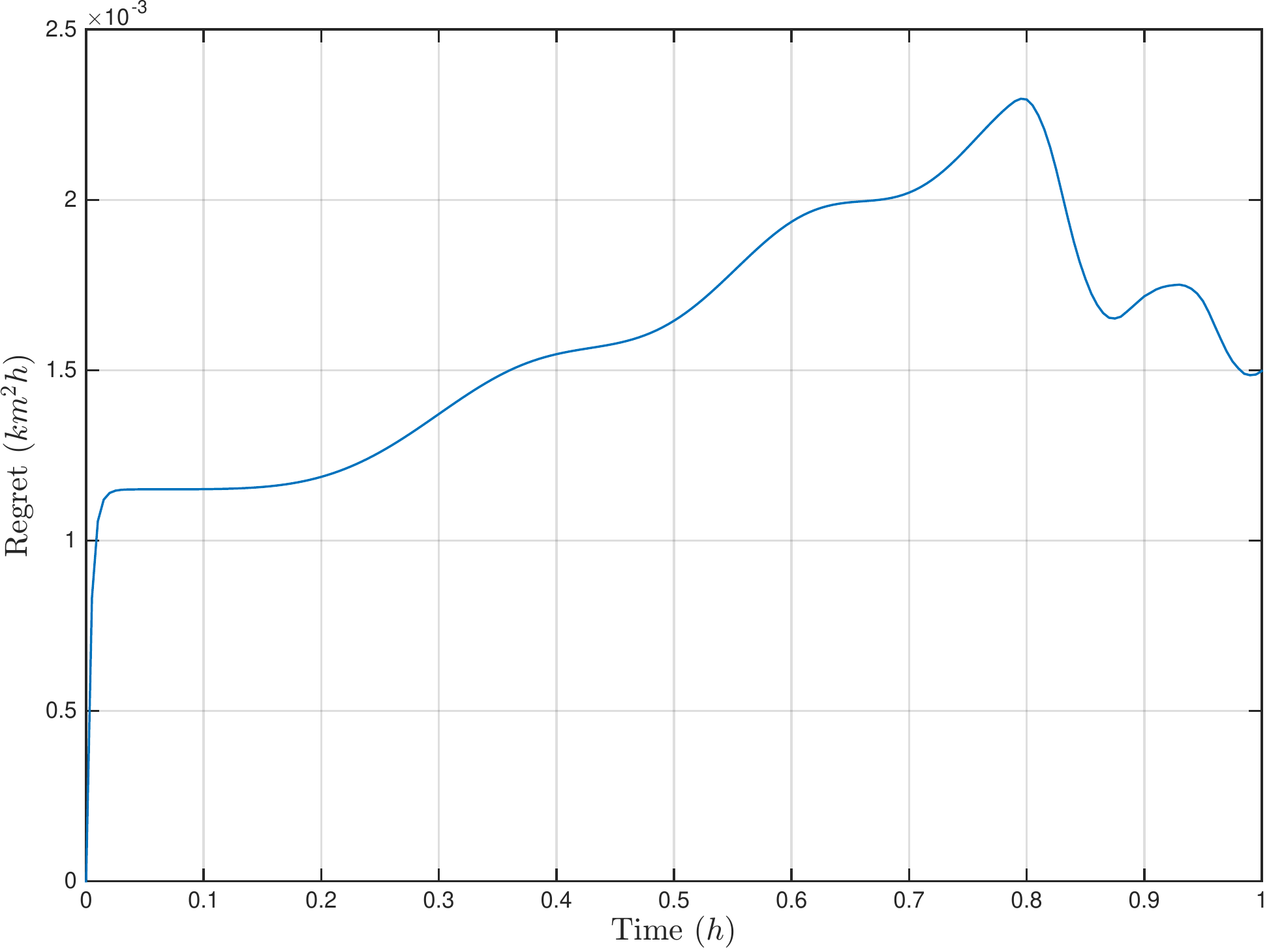}
\caption{Regret $\mathcal{R}_T$ for the preferred sheep problem (Section \ref{sec_preferred_sheep}) when the gain of the saddle point controller is set to be $\varepsilon=50$. The trajectory is strongly optimal, as predicted by Theorem \ref{theo:opti}, since the regret is bounded by a constant. The initial increment in the regret is due to the fact that the shepherd starts away from the first sheep while in the optimal offline trajectory would start close to it.}
\label{fig_regret_preferred_sheep}\end{figure}

%%%%%%%%%%%%%%%%%%%%%%%%%%%%%%%%%%%%%%%%%%%%%%%%%%%%%%%%%%%%%%%%%%%%%%%%%%%
%%%   F   I   G   U   R   E   %%%%%%%%%%%%%%%%%%%%%%%%%%%%%%%%%%%%%%%%%%%%%
%%%%%%%%%%%%%%%%%%%%%%%%%%%%%%%%%%%%%%%%%%%%%%%%%%%%%%%%%%%%%%%%%%%%%%%%%%%
%
\begin{figure}\centering
\includegraphics[width=\linewidth, height=0.62\linewidth]{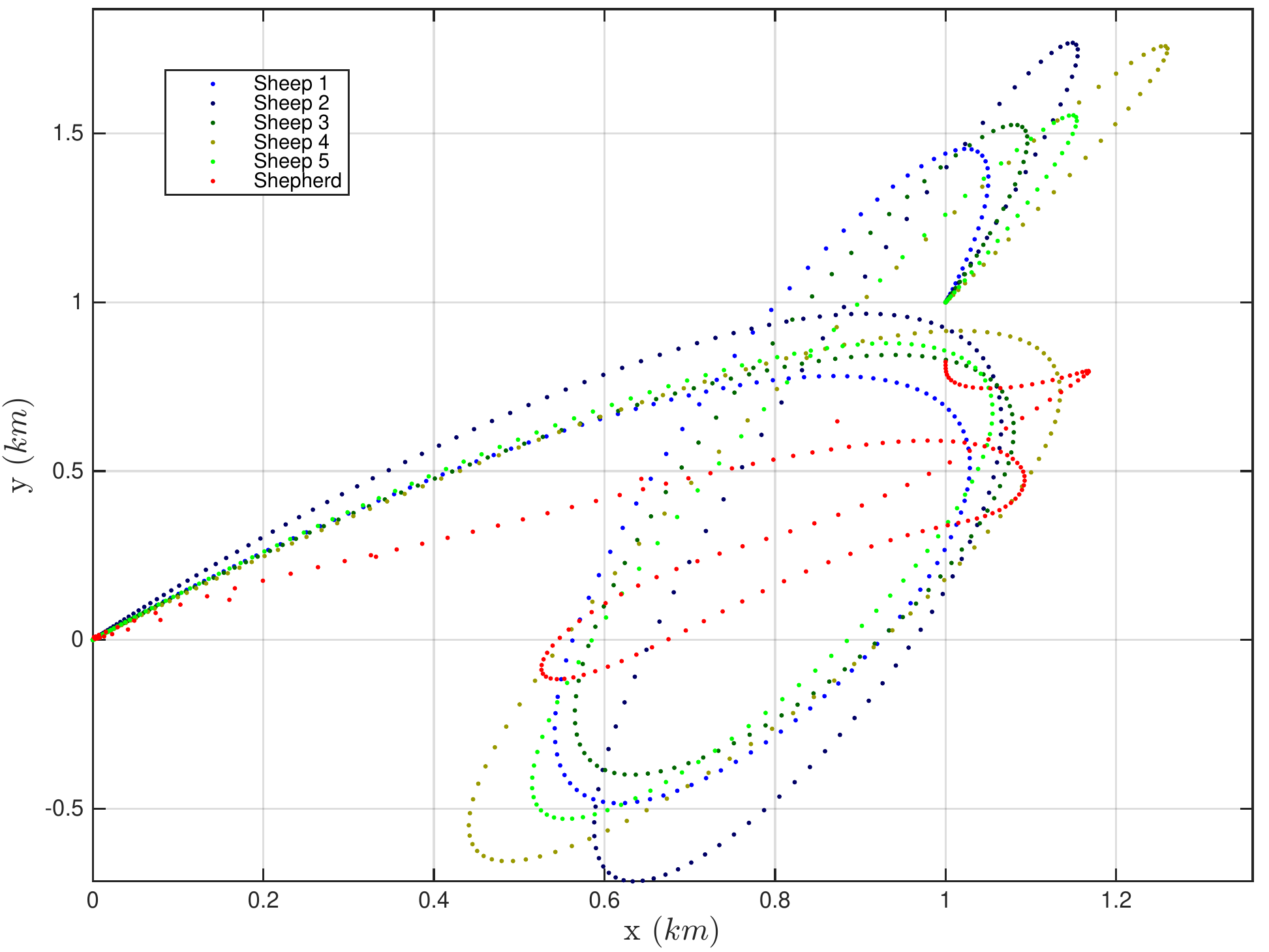}
\caption{Path of the sheep and the shepherd for the minimum acceleration problem (Section \ref{sec_minimum_acceleration}) when the gain of the saddle point controller is set to be $\varepsilon =50$. Observe that the shepherd path -- in red -- is not as close to the path of the sheep as in Figure \ref{fig:trajectory}. This is reasonable because the objective function and the constraints push the shepherd in different directions.}
\label{fig_trajectory_acceleration}\end{figure}

%%%%%%%%%%%%%%%%%%%%%%%%%%%%%%%%%%%%%%%%%%%%%%%%%%%%%%%%%%%%%%%%%%%%%%%%%%%
%%%   F   I   G   U   R   E   %%%%%%%%%%%%%%%%%%%%%%%%%%%%%%%%%%%%%%%%%%%%%
%%%%%%%%%%%%%%%%%%%%%%%%%%%%%%%%%%%%%%%%%%%%%%%%%%%%%%%%%%%%%%%%%%%%%%%%%%%
%
\begin{figure}\centering
\includegraphics[width=\linewidth, height=0.62\linewidth]{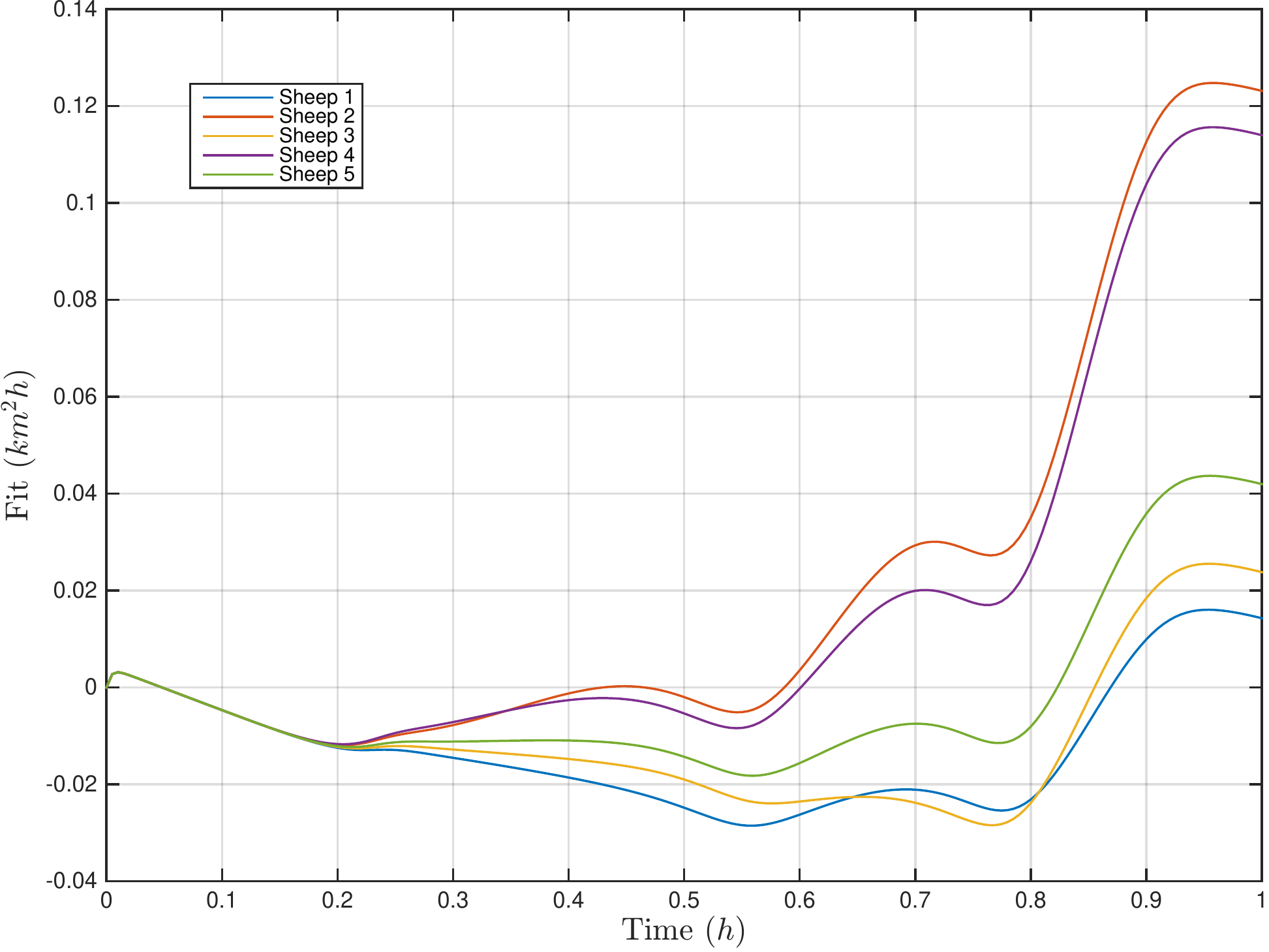}
\caption{Fit $\mathcal{F}_T$ for the minimum acceleration problem (Section \ref{sec_minimum_acceleration}) when the gain of the saddle point controller is set to $\varepsilon=50$. Since the fit is bounded, the trajectory is feasible in accordance with Theorem \ref{theo:opti}. Since the gradient of the objective function and the gradient of the feasibility constraints tend to point in different directions, the fit is larger than in the preferred sheep problem (c.f Figure \ref{fig_fit_preferred_sheep}).}
\label{fig_fit_acceleration}\end{figure}

%%%%%%%%%%%%%%%%%%%%%%%%%%%%%%%%%%%%%%%%%%%%%%%%%%%%%%%%%%%%%%%%%%%%%%%%%%%
%%%   F   I   G   U   R   E   %%%%%%%%%%%%%%%%%%%%%%%%%%%%%%%%%%%%%%%%%%%%%
%%%%%%%%%%%%%%%%%%%%%%%%%%%%%%%%%%%%%%%%%%%%%%%%%%%%%%%%%%%%%%%%%%%%%%%%%%%
%
\begin{figure}\centering
\includegraphics[width=\linewidth, height=0.62\linewidth]{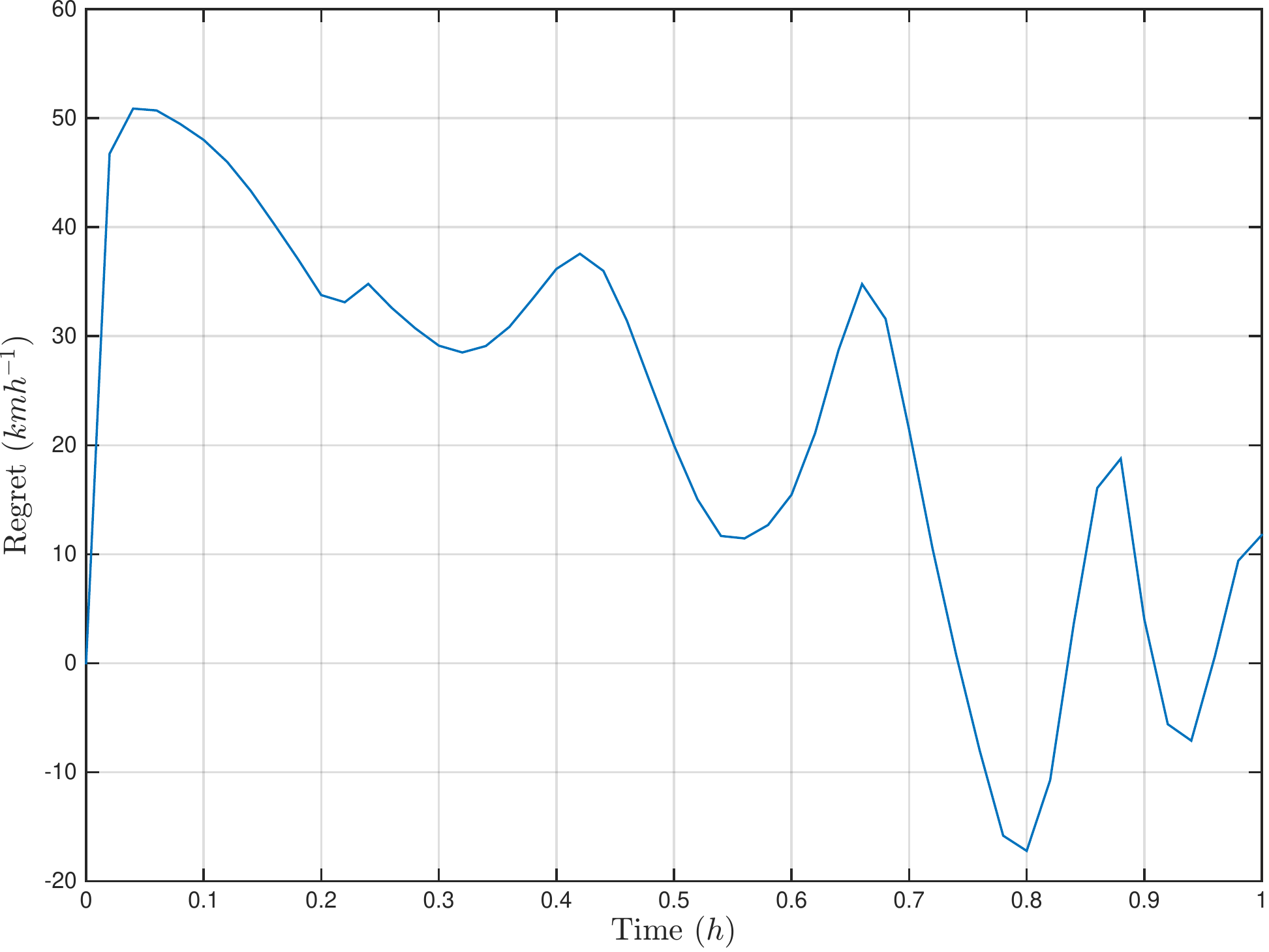}
\caption{Regret $\mathcal{R}_T$ for the minimum acceleration problem (Section \ref{sec_minimum_acceleration}) when the gain of the saddle point controller is set to be $\varepsilon=50$. The trajectory is strongly optimal as predicted by Theorem \ref{theo:opti}. Observe that regret is negative due to the fact that the agent is allowed to select different actions at different times as opposed to the clairvoyant player that is allowed to select a fixed action.}
\label{fig_regret_acceleration}\end{figure}
%
%%%%%%%%%%%%%%%%%%%%%%%%%%%%%%%%%%%%%%%%%%%%%%%%%%%%%%%%%%%%%%%%%%%%%%%%%%%%%%%% S U B S E C T I O N%%%%%%%%%%%%%%%%%%%%%%%%%%%%%%%%%%%%%%%%%%%%%%%%%%%%%%%%%%%%%%%%%%%%%%%%%%%%%%%%%%%%%%%%%%%%%%%%
%
\subsection{Saturated Fit}\label{sec_saturated_fit}
We apply the modified saddle point algorithm in the setting of Section \ref{sec_preferred_sheep} so to consider the saturated fit [c.f. \eqref{eqn_saturated_fit}] in lieu of the fit. Since the construction of the target trajectories gives a viable environment the hypotheses of Corollary \ref{corollary_saturated_fit2} are satisfied. Hence for a shepherd following the dynamics given by \eqref{eqn_action_descent} and \eqref{eqn_multiplier_ascent}, the trajectories are such that have saturated fit bounded by a function that grows sub linearly and bounded regret. For the simulation in this section the gain of the controller is set to $\varepsilon = 50$. Observe that the shepherd succeeds in following the herd, since his path remains close to the sheep (c.f. Figure \ref{fig_saturated_trajectory}). As predicted by the Corollary \ref{corollary_saturated_fit2} the fit of the trajectory is bounded by a function that grows sub linearly and the regret is bounded by a constant as it can be observed in figures \ref{fig_saturated_fit} and \ref{fig_saturated_regret} respectively. Further notice that the regret in this scenario is similar to the regret of the trajectory in the preferred sheep problem (c.f. Section \ref{sec_preferred_sheep}). 

%
%%%%%%%%%%%%%%%%%%%%%%%%%%%%%%%%%%%%%%%%%%%%%%%%%%%%%%%%%%%%%%%%%%%%%%%%%%%
%%%   F   I   G   U   R   E   %%%%%%%%%%%%%%%%%%%%%%%%%%%%%%%%%%%%%%%%%%%%%
%%%%%%%%%%%%%%%%%%%%%%%%%%%%%%%%%%%%%%%%%%%%%%%%%%%%%%%%%%%%%%%%%%%%%%%%%%%
%
\begin{figure}\centering
\includegraphics[width=\linewidth, height=0.62\linewidth]{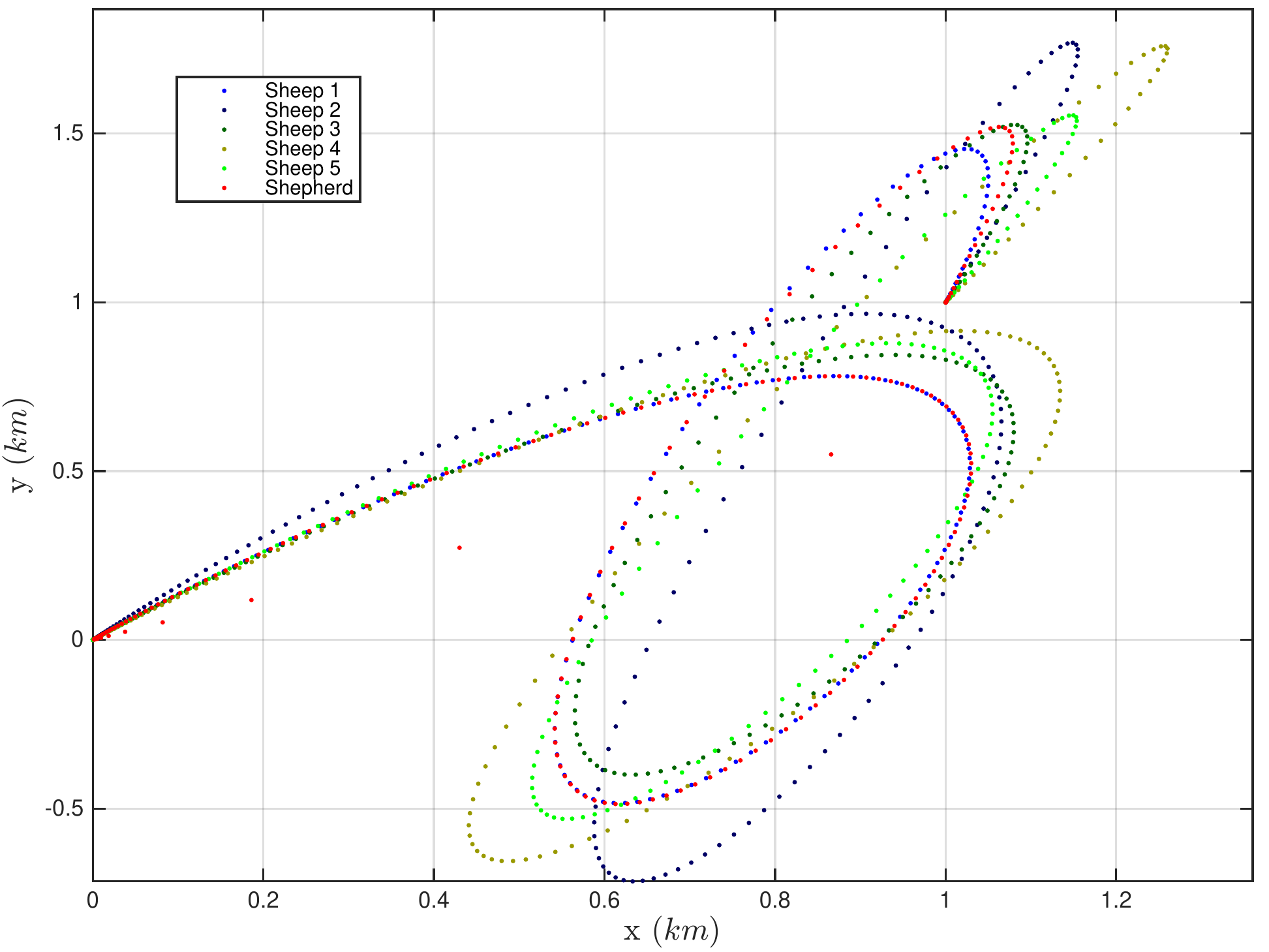}
\caption{Path of the sheep and the shepherd for preferred sheep problem when saturated fit is considered (Section \ref{sec_saturated_fit}) and the gain of the saddle point controller is set to be $\varepsilon =50$. The shepherd succeed in following the herd since its path -- in red -- is close to the path of all sheep.}
\label{fig_saturated_trajectory}\end{figure}

%%%%%%%%%%%%%%%%%%%%%%%%%%%%%%%%%%%%%%%%%%%%%%%%%%%%%%%%%%%%%%%%%%%%%%%%%%%
%%%   F   I   G   U   R   E   %%%%%%%%%%%%%%%%%%%%%%%%%%%%%%%%%%%%%%%%%%%%%
%%%%%%%%%%%%%%%%%%%%%%%%%%%%%%%%%%%%%%%%%%%%%%%%%%%%%%%%%%%%%%%%%%%%%%%%%%%
%
\begin{figure}\centering
\includegraphics[width=\linewidth, height=0.62\linewidth]{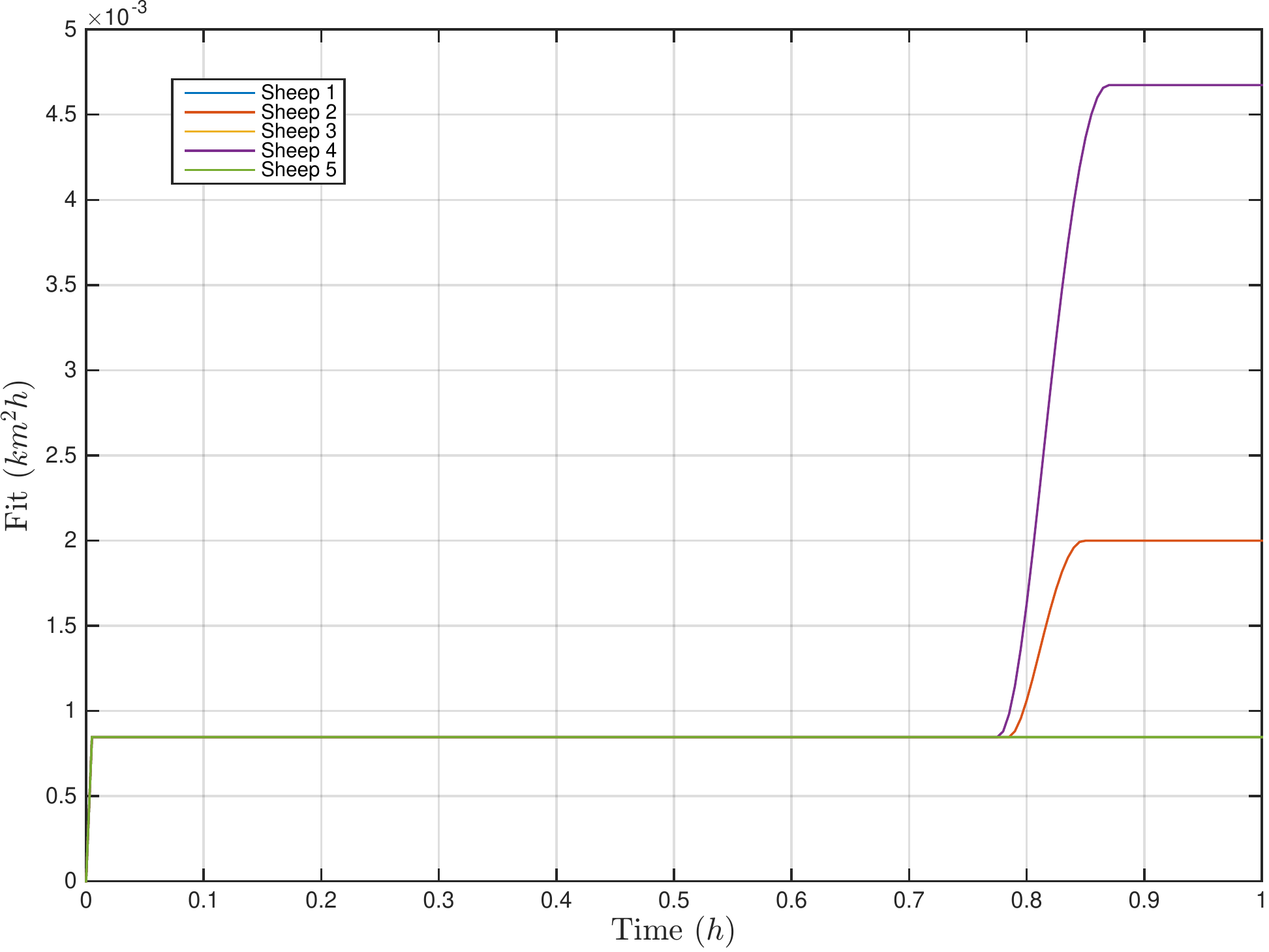}
\caption{Saturated fit $\mathcal{F}_T^{sat}$ for the preferred sheep problem (Section \ref{sec_saturated_fit}) when the gain of the saddle point controller is set to $\varepsilon=50$. Since the saturated fit grows sublinearly in accordance with Corollary \ref{corollary_saturated_fit2}, the trajectory is feasible. }
\label{fig_saturated_fit}\end{figure}

%%%%%%%%%%%%%%%%%%%%%%%%%%%%%%%%%%%%%%%%%%%%%%%%%%%%%%%%%%%%%%%%%%%%%%%%%%%
%%%   F   I   G   U   R   E   %%%%%%%%%%%%%%%%%%%%%%%%%%%%%%%%%%%%%%%%%%%%%
%%%%%%%%%%%%%%%%%%%%%%%%%%%%%%%%%%%%%%%%%%%%%%%%%%%%%%%%%%%%%%%%%%%%%%%%%%%
%
\begin{figure}\centering
\includegraphics[width=\linewidth, height=0.62\linewidth]{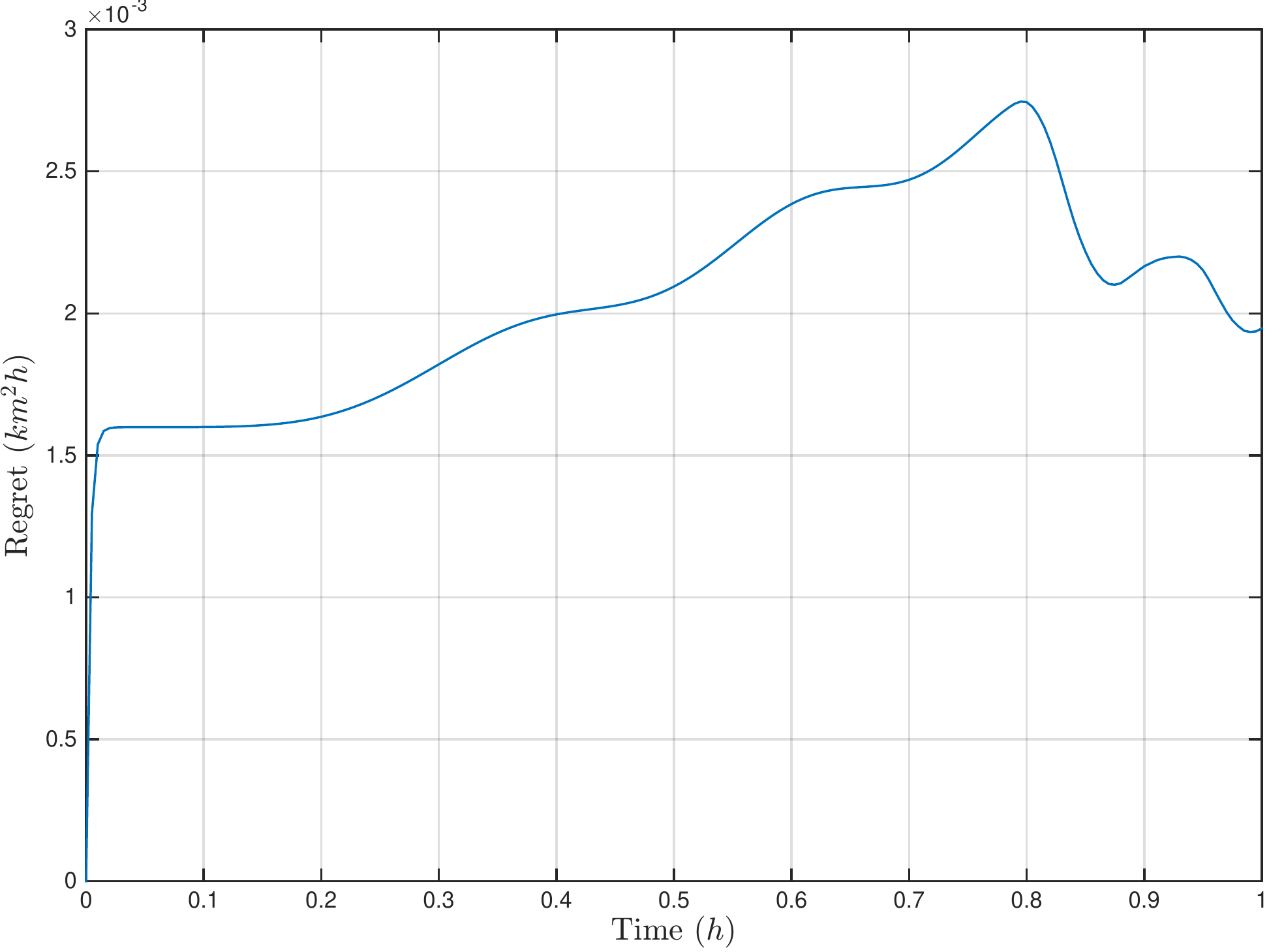}
\caption{Regret $\mathcal{R}_T$ for the preferred sheep problem when saturated fit is considered (Section \ref{sec_saturated_fit})and the gain of the saddle point controller is set to be $\varepsilon=50$. The regret is bounded as predicted by Corollary \ref{corollary_saturated_fit2} and therefore the trajectory is strongly optimal. Notice that regret in this case is identical to regret in the preferred sheep problem when regular fit is considered (c.f. Figure \ref{fig_regret_preferred_sheep}). }
\label{fig_saturated_regret}\end{figure}
%

%%%%%%%%%%%%%%%%%%%%%%%%%%%%%%%%%%%%%%%%%%%%%%%%%%%%%%%%%%%%%%%%%%%%%%%%%%%
%%%   S   E   C   T   I   O   N   %%%%%%%%%%%%%%%%%%%%%%%%%%%%%%%%%%%%%%%%%
%%%%%%%%%%%%%%%%%%%%%%%%%%%%%%%%%%%%%%%%%%%%%%%%%%%%%%%%%%%%%%%%%%%%%%%%%%%
%
\section{Conclusion}\label{sec_conclusions}
We considered a continuous time environment in which an agent must select actions to satisfy a set of constraints that are time varying and unknown a priori. We defined a viable environment as one in which there is a fixed action that satisfies the constraints at all times. We defined the fit as the cumulated constraint violation and the notions of feasible and strongly feasible trajectories. Feasible trajectories are such that the fit is bounded by a constant independent of the time horizon, and strongly feasible trajectories are such that the fit is bounded by a sublinear function of the time horizon. An objective function was considered to select a strategy that meets an optimality criterion and we defined regret in continuous time as the difference between the cumulative costs of the agent and the best clairvoyant agent. We then defined strongly optimal trajectories as those for which the regret is bounded by a constant that is independent of the time horizon. 

We proposed an online version of the saddle point controller of Arrow-Hurwicz to generate trajectories with small fit and regret. We showed that for any viable environment the trajectories that follow the dynamics of this controller are: (i) Strongly feasible if no optimality criterion is considered. (ii) Feasible and strongly optimal when an optimality criterion is considered. Numerical experiments on a shepherd that tries to follow a herd of sheep support these theoretical results. 

Future research includes studying asymptotic convergence of the saddle point dynamics to the optimal trajectory and studying systems with second order dynamics. In this setting, it is possible to add a term in the objective function that penalizes the action derivative, therefore allowing to control it and maintaining in a desired range.

%%%%%%%%%%%%%%%%%%%%%%%%%%%%%%%%%%%%%%%%%%%%%%%%%%%%%%%%%%%%%%%%%%%%%%%%%%%%
%%					APPENDIX
%%%%%%%%%%%%%%%%%%%%%%%%%%%%%%%%%%%%%%%%%%%%%%%%%%%%%%%%%%%%%%%%%%%%%%%%%%%%
%
\appendix
\section{Appendices}
\subsection{Proof of Lemma \ref{lemma:big_lemma}}\label{ap_big_lema_proof}
In order to develop this proof we need to define the tangent cone and to state Lemma \ref{lemma_tangent_cone} relating the projection of a vector over it and the projection over a convex set
\begin{definition}[Tangent cone]\label{def_tangent_cone}
Let $X \subset \reals^n$ be a closed convex set. We define the tangent cone to $X$ at $x_0$ as
\begin{equation}\label{eqn_tangent_cone} 
T_X(x_0) = \overline{\bigcup_{\theta > 0, x\in X}\theta(x-x_0)}.
\end{equation} 
\end{definition}
The above union is over all the points of the set $X$ and over all the positive reals $\theta$. Notice that the $\bigcup_{\theta>0}\theta(x-x_0)$ is the ray from $x_0$ and intersecting the point $x$. Thus, the tangent cone is the closure of the cone formed by all rays emanating from $x_0$ and intersecting at least one point $x\in X$ with $x\neq x_0$.
%
%
%
%\begin{lemma}\label{lemma_tangent_cone}
%For arbitrary $\delta$ and $v$ the projection over the set $X$ can be written as
%\begin{equation}
%P_X(x_0+\delta v) = x_0+\delta P_{T_{X}(x_0)}(v)+\mathcal{O}(\delta ),
%\end{equation}
%where $\mathcal{O}(\delta)$ is a function such that $\lim_{\delta \to 0} \mathcal{O}(\delta )/ \delta =0$.
%\end{lemma}
%\begin{proof}
%See \cite{Zhang95} Lemma 4.6 page 300.
%\end{proof}
%
\begin{lemma}\label{lemma_tangent_cone}
Let $X \in \reals^n$ be a closed convex set, let $x_0\in X$ and let $v \in \reals^n$. Then the projection of v over the set $X$ at $x_0$ defined in \eqref{eqn_proj_over_X} is
\begin{equation}
\Pi_X(x_0,v) = P_{T_X(x_0)}(v).
\end{equation}
\end{lemma}
\begin{proof}
The proof follows from Lemma 4.6 in \cite{Zhang95}.
\end{proof}
\begin{proof}[Proof of Lemma \ref{lemma:big_lemma}]
Consider the case in which $x_0 \in int(X)$. Then, for any $v$ there exits a small enough $\delta > 0$ such that $x_0+\delta v \in X$. Hence $P_X(x_0 + \delta v) = x_0+\delta v$ and it holds that
\begin{equation}
P_X(x_0+\delta v) - x_0  = v\delta .
\end{equation}
Thus $\Pi_X(x,v) = v$ and \eqref{eqn_big_lemma} is verified.
When $x_0\in \partial X$ two cases are possible; either $x_0 + \delta v \in T_X(x_0)$ for small enough $\delta >0$ or $x_0+ \delta v \notin T_X(x_0)$ for all $\delta >0$. In the first case because of Lemma \ref{lemma_tangent_cone} it is verified that
\begin{equation}
\Pi_X(x_0,v) = P_{T_X(x_0)}(v) = v.
\end{equation}
And therefore \eqref{eqn_big_lemma} holds. Let us now consider the last case in which $x_0 \in \partial X$ and $x_0 +\delta v \notin T_X(x_0)$. Because X is a convex set there exists a vector $a \in \reals^n$ with $\|a \| =1$ defining a supporting hyperplane at $x_0$ 
%
%\begin{equation}
$\mathcal{H} = \{ x\in \reals^n : a^T(x-x_0) = 0 \},$
%\end{equation}
%
and for all $x \in X$ we have that
\begin{equation}\label{eqn_supporting_hyperplane}
a^T(x-x_0) \leq 0.
\end{equation}
If the set $X$ is smooth at $x_0$ then the border of the tangent cone at the point $x_0$ is contained in the hyperplane $\mathcal{H}$, therefore $\Pi_X(x_0,v) \subset \mathcal{H}$. Thus, $a^T \Pi_X(x_0,v) = 0$ and we have as well that $a^T v \geq 0$, otherwise there must exists a $\delta > 0$ such that $x_0 + \delta v \in T_X(x_0)$. On the other hand if there is a corner at $x_0$ there are infinite supporting hyperplanes.  One of them verifies that $a^T v \geq 0$ and contains the boundary of the tangent cone, thus $a^T \Pi_X(x_0,v) = 0$. Since $\Pi_X(x_0,v)$ is the projection of $v$ over the tangent cone, we have that: $\Pi_X(x_0,v) = P_{T_X(x_0)}(v) = (a_\perp^T v) a_\perp$, where $a_\perp \in \reals^n$ and verifies that $a^T a_\perp =0$ and $\| a_\perp \|=1$. Projecting the vectors $x_0 -x$ and $v$ over $a$ and $a_\perp$, we have
\begin{equation}
(x_0 -x)^Tv = (x_0-x)^Tav^Ta + (x_0-x)^Ta_\perp v^Ta_\perp .
\end{equation}
From the previous discussion the above equation reduces to
\begin{equation}
(x_0 - x)^T v = (x_0-x)^T a v^T a +(x_0-x)^T \Pi_X(x_0,v).
\end{equation}
By combining the fact that $v^T a \geq 0$ and \eqref{eqn_supporting_hyperplane} the left hand side of the above equality can be lower bounded by
\begin{equation}
(x_0 -x)^T v \geq (x_0 -x )^T \Pi_X (x_0,v).
\end{equation}
Hence we have proved the lemma for all posible cases. 
\end{proof}

%%%%%%%%%%%%%%%%%%%%%%%%%%%%%%%%%%%%%%%%%%%%%%%%%%%%%%%%%%%%%%%%%%%%%%%%%%%%
%%					APPENDIX
%%%%%%%%%%%%%%%%%%%%%%%%%%%%%%%%%%%%%%%%%%%%%%%%%%%%%%%%%%%%%%%%%%%%%%%%%%%%
%
\subsection{Proof of Lemma \ref{lemma:regret_lower_bound}}\label{ap_regret_lower_bound}
%
%Because of Assumption \ref{as:lower_bound}, we have that
%
%\begin{equation}
% \min_{x\in X}f_0(t,x) - f_0(t,x^*) \geq -K. 
%\end{equation}
%
Let $x(t)$ be the action at time $t$ when the agent follows the dynamics defined by \eqref{eqn_action_descent} and \eqref{eqn_multiplier_ascent}, because of Assumption \ref{as:lower_bound}, we have that
\begin{equation}
f_0(t,x(t)) - f_0(t,x^*) \geq -K, 
\end{equation} 
Integrating both sides of the above equation yields
\begin{equation}
\int_{0}^T f_0(t,x(t))dt - \int_{t=0}^Tf_0(t,x^*) dt\geq -KT.
\end{equation}
Since the left hand side of the above equation is the regret up to time $T$ defined in \eqref{eqn_continuous_regret}, the proof is completed. 
%
%%%%%%%%%%%%%%%%%%%%%%%%%%%%%%%%%%%%%%%%%%%%%%%%%%%%%%%%%%%%%%%%%%%%%%%%%%%%%%%%%%%%%
%%%%%%%%%%%%%%%%%%%%%%%%%% S E C T I O N %%%%%%%%%%%%%%%%%%%%%%%%%%%%%%%%%%%%%%%%%%%%%%%%%%%
%%%%%%%%%%%%%%%%%%%%%%%%%%%%%%%%%%%%%%%%%%%%%%%%%%%%%%%%%%%%%%%%%%%%%%%%%%%%%%%%%%%%%
\subsection{Proof of Theorem \ref{theo:opti}}\label{ap_theo_opti}

Consider action trajectories $x(t)$ and multiplier trajectories $\lambda(t)$ and the  corresponding energy function $V_{\bbarx,\bar{\lambda}}(x(t),\lambda(t))$ in \eqref{eqn_lyapunov}, for arbitrary given action $\bbarx \in \reals^n$ and multiplier $\bar{\lambda}\in \Lambda$. The derivative $\dot V_{\bbarx,\bar{\lambda}}(x(t),\lambda(t))$ is given by
\begin{equation}\label{eqn_theo_both_pf_10}
   \dot{V}_{\bbarx,\bar{\lambda}} (x(t),\lambda(t))
      = (x(t) - \bbarx)^T\dot{x}(t) + (\lambda(t) -\bar{\lambda})^T\dot{\lambda}(t).
\end{equation}
If the trajectories $x(t)$ and $\lambda(t)$ follow from the saddle point dynamical system defined by \eqref{eqn_action_descent} and \eqref{eqn_multiplier_ascent} respectively we can substitute the action and multiplier derivatives by their corresponding values and reduce \eqref{eqn_theo_both_pf_10} to
%
%\begin{align}\label{eqn_theo_survival_opti_pf_11}
%\dot{V}_{\bbarx,\bar{\lambda}}(x(t),\lambda(t))= (x(t) - \bbarx )^T \Pi_X ( x, -\varepsilon( f_{0,x}(t,x(t)) \nonumber \\ 
%+f_x(t,x(t))\lambda(t) ) 
% +(\lambda(t)-\bar{\lambda})^T \Pi_{\Lambda} (\lambda,\varepsilon f(t,x(t))).
%\end{align}
%
\begin{align}\label{eqn_theo_survival_opti_pf_11}
\dot{V}_{\bbarx,\bar{\lambda}}(x(t),\lambda(t)) &= (x(t) - \bbarx )^T \Pi_X ( x, -\varepsilon \ccalL_{x}(t,x(t),\lambda(t))) \nonumber \\ 
 &+(\lambda(t)-\bar{\lambda})^T \Pi_{\Lambda} (\lambda,\varepsilon \ccalL_{\lambda}(t,x(t),\lambda(t))).
\end{align}
Then, use Lemma \ref{lemma:big_lemma} for both $X$ and $\Lambda$ to write
%
%\begin{align}\label{eqn_theo_survival_opti_pf_11}
%\dot{V}_{\bbarx,\bar{\lambda}} (x(t),\lambda(t))&\leq \varepsilon [ -(x(t)-\bbarx)^T(f_{0,x}(t,x(t))  \\ \nonumber
%&+f_x(t,x(t))\lambda(t)) 
% +(\lambda(t)-\bar{\lambda})^T f(t,x(t))]. 
%\end{align}
%
\begin{align}\label{eqn_theo_survival_opti_pf_11}
\dot{V}_{\bbarx,\bar{\lambda}} (x(t),\lambda(t))&\leq -\varepsilon (x(t)-\bbarx)^T\ccalL_{x}(t,x(t),\lambda(t))  \\ \nonumber
&+\varepsilon(\lambda(t)-\bar{\lambda})^T \ccalL_{\lambda}(t,x(t),\lambda(t)). 
\end{align}
%
%Notice that $\mathcal{L}(t,x(t),\lambda(t)) = f_0(t,x(t)) +\lambda(t)^T f(t,x(t))$ is a convex function with respect to the actions since it is a sum of convex functions with respect to $x$. Then, using the definition of subgradient (c.f. Definition \ref{def_subgradient}) we can upper bound the inner product 
Since $\mathcal{L}(t,x(t),\lambda(t)) $ is a convex function, \eqref{eqn_def_subgradient} takes the form
%
%\begin{equation}
%\begin{split}
%-(x(t)-\bar{x} )^T(f_{0,x}(t,x(t)) + f_x(t,x(t))\lambda(t) ) \\
%= -(x(t)-\bar{x} )^T \mathcal{L}_x(t,x(t),\lambda(t))
%\end{split}
%\end{equation}
%
%
\begin{equation}\label{eqn_auxiliar_equation_theorem}
-(x(t)-\bar{x} )^T\ccalL_{x}(t,x(t),\lambda(t))  \leq  \mathcal{L}(t,\bar{x},\lambda(t)) - \mathcal{L}(t,x(t),\lambda(t)).
\end{equation}
From the linearity of the Lagrangian with respect to $\lambda$ we have
\begin{equation}\label{eqn_auxiliar_equation_theorem2}
(\lambda(t)-\bar{\lambda})^T \ccalL_{\lambda}(t,x(t),\lambda(t)) = \ccalL (t,x(t),\lambda(t)) - \ccalL(t,x(t),\bar{\lambda}).
\end{equation}
Combine expressions \eqref{eqn_auxiliar_equation_theorem} and \eqref{eqn_auxiliar_equation_theorem2} to reduce \eqref{eqn_theo_survival_opti_pf_11} to
\begin{equation}
\dot{V}_{\bbarx,\bar{\lambda}}(x(t),\lambda(t)) \leq \varepsilon\left(\mathcal{L}(t,\bar{x},\lambda(t)) - \ccalL(t,x(t),\bar{\lambda})\right).
\end{equation}
%Then, we can upper bound the right hand side of the equation \eqref{eqn_theo_survival_opti_pf_11} and obtain
%by the difference $\mathcal{L}(t,\bar{x},\lambda(t)) - \mathcal{L}(t,x(t),\lambda(t))$. Then, we can upper bound the right hand side of the equation \eqref{eqn_theo_survival_opti_pf_11} and obtain
%\begin{align}
%\dot{V}_{\bbarx,\bar{\lambda}}(x(t),\lambda(t))&\leq\varepsilon[ f_0({t,\bbarx})+\lambda^T(t) f(t,\bbarx)-f_0(t,x(t)) \nonumber \\
%&-\lambda^T(t)f(t,x(t)) +(\lambda(t)-\bar{\lambda})^Tf(t,x(t))].
%\end{align}
%
%
%%%%%%%%%%%%%%%%%%%%%%%%%%%%%%%%%%%%%%%%%%%%%%%%%%%%%%%%%%%%%%%%%%%%%%%%%%%%%%%%%%%%%%%%%%%
%%%%%%%%%%%%				NOT DISPLAYED
%According to \eqref{eqn_theo_survival_pf_13} we have:
%\begin{align}
%\dot{V}(x(t),\lambda(t))_{\bbarx,\bar{\lambda}} \leq \varepsilon[f_0(t,\bbarx)+\lambda^T(t) f(t,\bbarx) %-f_0(t,x(t)) \nonumber \\
%-\lambda^T(t) f(t,x(t)) +(\lambda(t)  -\bar{\lambda}) ^T f(t,x(t))]
%\end{align}
%%%%%%%%%%%%%%%%%%%%%%%%%%%%%%%%%%%%%%%%%%%%%%%%%%%%%%%%%%%%%%%%%%%%%%%%%%%%%%%%%%%%%%%%%%%%%
%
%
Substituting the Lagrangians by the expression \eqref{eqn_lagrangian}
\begin{align}
\dot{V}_{\bbarx,\bar{\lambda}}(x(t),\lambda(t))\leq \varepsilon[f_0(t,\bbarx) + \lambda^T (t)f(t,\bbarx) \nonumber \\
-f_0(t,x(t)) - \bar{\lambda}^T f(t,x(t))]. 
\end{align}
Rewriting the above inequality and integrating both sides with respect to the time from time $t=0$ to $t=T$, we obtain
\begin{align}
\int_0^T f_0(t,x(t))-f_0(t,\bbarx)+\bar{\lambda}^T f(t,x(t)) - \lambda^T (t)f(t,\bbarx) dt \nonumber\\ 
\leq -\frac{1}{\varepsilon}\int_0^T \dot{V}_{\bbarx,\bar{\lambda}}(x(t),\lambda (t))   dt .
\end{align}
Using the result \eqref{eqn_inequality_chain} the above equation reduces to yields
\begin{align}
\int_0^T f_0(t,x(t))-f_0(t,\bbarx)&+\bar{\lambda}^T f(t,x(t)) - \lambda^T(t) f(t,\bbarx) dt \nonumber\\ &
\leq \frac{1}{\varepsilon}V_{\bbarx,\bar{\lambda}}(x(0),\lambda(0)).
\label{eqn_fundamental}
\end{align}
Since \eqref{eqn_fundamental} holds for any $\bar{x} \in X$ and any $\bar{\lambda} \in \Lambda$, it holds for $\bar{x} = x^*$, $\bar{\lambda} = 0$. Since $\lambda^T(t)f(t,x^*) \,dt \leq 0 $ $\; \forall t\in[0,T]$ we can lower bound the left hand side of \eqref{eqn_fundamental} to obtain:
\begin{equation}
\int_0^Tf_0(t,x(t)) - f_0(t,x^*) dt \leq \frac{1}{\varepsilon} V_{x^*,0}(x(0),\lambda(0)).
\end{equation}
Notice that the left hand side of the above equation is the definition of regret given in \eqref{eqn_continuous_regret}. Thus, we have showed that \eqref{eqn_regret_upper_bound_full_problem} holds and since the right hand side of the above equation is a constant for all $T $ we proved that the trajectory generated by the saddle point controller is strongly optimal. It remains to prove that the trajectory generated is feasible. Choosing $\bar{x} = x^*$ in \eqref{eqn_fundamental} and using the result of Lemma \ref{lemma:regret_lower_bound} yields
\begin{align}
\int_0^T\bar{\lambda}^Tf(t,x(t)) -& \lambda^T(t)f(t,x^*) \,dt  \nonumber \\
&\leq \frac{1}{\varepsilon} V_{x^*,\bar{\lambda}}(x(0),\lambda(0)) +KT.
\end{align}
Since $\lambda^T(t)f(t,x^*) \,dt \leq 0  \; \forall t\in[0,T]$ the left hand side of the above equation is lower bounded by $\bar{\lambda}^T\int_0^Tf(t,x(t))$, yielding 
\begin{equation}\label{eqn_algun_numero}
\bar{\lambda}^T\int_0^Tf(t,x(t)) dt \leq \left( V_{x^*,\bar{\lambda}} (x(0),\lambda(0))\right) / \varepsilon +KT.
\end{equation}
Now let's choose $\bar{\lambda} = \left[ \ccalF_T\right]^+= \left[\int_0^Tf(t,x(t)) \,dt \right]^+$. Let $I = \{i=1..m | \int_0^T f_i(t,x(t)) \,dt \geq 0) \}$. Notice that if $i \not\in I$, then $\bar{\lambda}_i\int_0^Tf_i(t,x(t)) \,dt =0$. On the other hand, if $i \in I$, $\bar{\lambda}_i\int_0^Tf_i(t,x(t)) \,dt = \left( \int_0^Tf_i(t,x(t)) \,dt \right)^2 \geq 0$. Thus,
\begin{equation}
\bar{\lambda}^T\int_0^Tf(t,x(t)) dt = \left\| \left[ \ccalF_T\right]^+\right\|^2.
\end{equation}
Write then inequality \eqref{eqn_algun_numero} for the particular choice of $\bar{\lambda}$ as
\begin{equation}
\left\| \left[ \ccalF_T\right]^+\right\|^2 \leq \frac{1}{\varepsilon} V_{x^*,\left[\ccalF_T\right]^+}(x(0),\lambda(0)) +KT.
\end{equation}
Use the definition of the energy function $V_{\bar{x},\bar{\lambda}} \left(x,\lambda\right)$ given in \eqref{eqn_lyapunov} to write the above inequality as
\begin{equation}
\left\| \left[\ccalF_T\right]^+ \right\|^2 \leq \frac{1}{\varepsilon}\left( \left\|x(0)-x^*\right\|^2 +\left\|\left[\ccalF_T\right]^+ - \lambda(0)\right\|^2\right)+KT.
\end{equation}
Expand the second square in the right hand side of the above expression and re arrange terms to write
\begin{equation}
\begin{split}
&\left\| \left[\ccalF_T\right]^+ \right\|^2 +  \lambda^T(0)\left[\ccalF_T\right]^+\frac{2}{\varepsilon-1} \\
&\leq \frac{1}{\varepsilon-1}\left( \left\|x(0)-x^*\right\|^2 +\left\|\lambda(0) \right\|^2 \right)+KT\frac{\varepsilon}{\varepsilon-1}.
\end{split}
\end{equation}
Adding in both sides of the above inequality $\left\| \lambda(0) \right\|^2 \left(\frac{1}{\varepsilon-1}\right)^2$, then factorizing the left hand side the above inequality yields
\begin{equation}
\begin{split}
\left\|\left[\ccalF_T\right]^+  + \lambda(0) \frac{1}{\varepsilon-1}\right\|^2 &\leq\frac{1}{\varepsilon-1}\left\|x(0)-x^*\right\|^2  +KT\frac{\varepsilon}{\varepsilon-1}\\
&+\frac{\left\|\lambda(0) \right\|^2}{\varepsilon-1}\left(1+\frac{1}{\varepsilon-1}\right).
\end{split}
\end{equation}
%
%Taking the square root on both sides of the above inequality and isolating $\ccalF_T$ yields
%
%\begin{equation}
%\begin{split}
%&\ccalF_T \leq -\lambda(0) \frac{\varepsilon}{\varepsilon^2-1} \\
%&\sqrt{\frac{\varepsilon}{\varepsilon^2-1}\left(x(0)-x^*\right)^2  +KT\frac{\varepsilon^2}{\varepsilon^2-1}+\left(\frac{\varepsilon}{\varepsilon^2-1}\right)\left(1+\frac{\varepsilon}{\varepsilon^2-1}\right)\left\|\lambda(0) \right\|^2}.
%\end{split}
%\end{equation}
%%%%%%%%%%%%%%%%%%%%%%%%%%%%%%%%%%%%%%%%%%%%%%%%%%%%%%%%%%%%%%%%%%%
%
%Notice that, the left hand side of the above equation is the square of the $i$th component of the fit. Thus for all $i \in I$ it is clear that:
%\begin{equation}\label{eqn_penalty_bound}
%\mathcal{F}_{T,i} \leq \left(\frac{1}{\varepsilon}V_{x^*,\left[\int_0^Tf(t,x(t)) \,dt \right]^+}(x(0),\lambda(0)) +KT\right)^{1/2}.
%\end{equation}
%
Since the term $\lambda(0)/\left(\varepsilon-1\right)$ is constant with respect to $T$ it is the case that the norm of $\left[\ccalF_T\right]^+$ is bounded by a function that grows like $\sqrt{T}$. On the other hand it also holds that $\|\left[\ccalF_T\right]^+\|$ is bounded by a constant function of the gain $\varepsilon$. These observations lead to the conclusion that
\begin{equation}
\| \left[ \ccalF_T\right]^+\| \leq \ccalO\left(\sqrt{KT},\varepsilon^0\right).
\end{equation}
The above inequality implies that for any $i\in I$ it is the case that $\ccalF_{T,i} \leq \ccalO\left(\sqrt{KT},\varepsilon^0\right)$.
If $i \not\in I$ it means that $\mathcal{F}_{T,i} < 0$ and it trivially satisfies \eqref{eqn_penalty_bound}. Which proves that the trajectories that are solution of the saddle point controller defined by \eqref{eqn_action_descent} and \eqref{eqn_multiplier_ascent} are feasible since they are bounded by a sublinear function of the time horizon for all $T$.
\bibliographystyle{ieeetr}
\bibliography{bib}
\begin{IEEEbiography}{Santiago Paternain} %[{\includegraphics[width=1in,height=1.25in,clip,keepaspectratio]{filename}}]  
received the B.Sc. degree in electrical engineering from Universidad de la Rep\'ublica Oriental del Uruguay, Montevideo, Uruguay in 2012. Since August 2013, he has been working toward the Ph.D. degree in the Department of Electrical and Systems Engineering, University of Pennsylvania. His research interests include optimization and control of dynamical systems. 
\end{IEEEbiography}
\begin{IEEEbiography}{Alejandro Ribeiro} received the B.Sc. degree in electrical engineering from the Universidad de la Rep\'ublica Oriental del Uruguay, Montevideo, in 1998 and the M.Sc. and Ph.D. degree in electrical engineering from the Department of Electrical and Computer Engineering, the University of Minnesota, Minneapolis in 2005 and 2007. From 1998 to 2003, he was a member of the technical staff at Bellsouth Montevideo. After his M.Sc. and Ph.D studies, in 2008 he joined the University of Pennsylvania (Penn), Philadelphia, where he is currently the Rosenbluth Associate Professor at the Department of Electrical and Systems Engineering. His research interests are in the applications of statistical signal processing to the study of networks and networked phenomena. His current research focuses on wireless networks, network optimization, learning in networks, networked control, robot teams, and structured representations of networked data structures. Dr. Ribeiro received the 2012 S. Reid Warren, Jr. Award presented by Penn's undergraduate student body for outstanding teaching, the NSF CAREER Award in 2010, and student paper awards at the 2013 American Control Conference (as adviser), as well as the 2005 and 2006 International Conferences on Acoustics, Speech and Signal Processing. Dr. Ribeiro is a Fulbright scholar and a Penn Fellow
\end{IEEEbiography}

\end{document}